\documentclass[a4paper,12pt,reqno]{amsart}
\usepackage{lmodern,amsmath,nccmath,amssymb,mathtools,amsthm}
\usepackage{enumerate,cite}
\usepackage{graphics,graphicx,geometry}
\usepackage{stackengine}
\usepackage{hyperref}
\usepackage{xcolor}
\usepackage{mathrsfs}
\usepackage{wasysym}
\usepackage{cleveref}
\usepackage{floatrow,morefloats}
\usepackage{epsfig}
\usepackage{epstopdf}
\usepackage{caption,subcaption}
\usepackage{exscale, relsize}
\usepackage{parskip}
\usepackage{hyphenat}

\numberwithin{equation}{section}

\newtheorem{definition}{Definition}
\newtheorem{problem}{Problem}
\newtheorem{lemma}{Lemma}
\newtheorem{theorem}{Theorem}
\newtheorem{proposition}{Proposition}
\newtheorem{corollary}{Corollary}
\newtheorem{remark}{Remark}
\newtheorem{example}{Example}

\numberwithin{equation}{section}
\numberwithin{definition}{section}
\numberwithin{lemma}{section}
\numberwithin{theorem}{section}
\numberwithin{corollary}{section}
\numberwithin{remark}{section}
\numberwithin{example}{section}

\title{Inequalities involving a measure of  Marcell\'{a}n class and zeros of corresponding orthogonal polynomials}

\author{VIKASH KUMAR$^\dagger$}
\address{$^\dagger$Department of Mathematics\\ Indian Institute of Technology, Roorkee-247667, Uttarakhand, India}
\email{vikaskr0006@gmail.com, vkumar4@mt.iitr.ac.in}

\author{A. Swaminathan$^\ddagger$}
\address{$^\ddagger$Department of Mathematics\\ Indian Institute of Technology, Roorkee-247667, Uttarakhand, India}
\email{mathswami@gmail.com, a.swaminathan@ma.iitr.ac.in}
\allowdisplaybreaks
\bigskip
\begin{document}


	\begin{abstract}
			Let $\tilde{\Phi}_n$ be a quasi-orthogonal polynomial of order 1 on the unit circle, obtained from an orthogonal polynomial $\Phi_n$ with measure $\mu$, which is in the Marcell\'{a}n class, if there exist another measure $\tilde{\mu}$ such that $\tilde{\Phi}_n$ is a monic orthogonal polynomial. This article aims to investigate various properties related to the Marcell\'{a}n class. At first, we study the behaviour of the zeros between $\Phi_n$ and $\tilde{\Phi}_n$. Along with numerical examples, we analyze the zeros of $\Phi_n$, its POPUC and the linear combination of the POPUC. Further, comparison of the norm inequalities among $\Phi_n$ and $\tilde{\Phi}_n$ are obtained by involving their measures. This leads to the study of the Lubinsky type inequality between the measures $\mu$ and $\tilde{\mu}$, without using the ordering relation between $\mu$ and $\tilde{\mu}$. Additionally, similar type of inequalities for the  kernel type polynomials related to $\mu$ and $\tilde{\mu}$ are obtained.
	\end{abstract}
	\maketitle
	\markboth{Vikash Kumar and A. Swaminathan}{Marcell\'{a}n Class}

	\section{Introduction}
	Suppose $\{\Phi_{n}\}$ be the sequence of monic orthogonal polynomials with respect to the non-trivial positive Borel measure $\mu$ on the unit circle $\partial\mathbb{D}=\{z\in \mathbb{C}; \lvert z\rvert=1\}$. In other words, the polynomials satisfy the orthogonality condition:
	\begin{align*}
		\int_{\partial\mathbb{D}} \Phi_{n}(z)\overline{\Phi_{m}(z)}d\mu(z)=k_n^{-2}\delta_{nm},
	\end{align*}
where $k_n^{-1}=\lVert \Phi_{n}\rVert_{\mu}$,  where $\lVert\cdot\rVert_\mu$ represents the $L^2(\partial\mathbb{D},d\mu)$ norm.  Through the use of a suitable sequence $\{\alpha_n\}$, where each $\alpha_n$ belongs to the unit disk $\mathbb{D}$, we can recursively determine the monic orthogonal polynomials $\Phi_{n}$. This process is achieved using the Szeg\"{o} recurrence relations:
\begin{align}\label{Szego recursion}
\nonumber	\Phi_{n}(z) &= z\Phi_{n-1}(z) - \bar{\alpha}_{n-1}\Phi_{n-1}^*(z),\\
	\Phi_{n}^*(z) &= \Phi_{n-1}^*(z) - \alpha_{n-1}z\Phi_{n-1}(z),
\end{align}
where the initial conditions are set as $\Phi_{0}(z) = 1$, $\Phi_{-1}(z) = 0$, and $\Phi_{n}^*(z) = z^n\overline{\Phi_{n}(\frac{1}{\bar{z}})}$ are known as reversed polynomial. Verblunsky theorem guarantees the existence of Verblunsky coefficients ${\alpha_{n}}$. It is worth mentioning that these coefficients are sometimes referred to as reflection or Schur parameters\cite{Zhedanov_SIGMA_2020}. The polynomial $\Phi_n(z)$ can be normalized to obtain the orthonormal polynomial denoted by $\phi_n(z)=\frac{\Phi_{n}(z)}{\lVert\Phi_{n}\rVert_\mu}$, where $\lVert\cdot\rVert_\mu$ represents the $L^2(\partial\mathbb{D},d\mu)$ norm.

For several decades, the study of finite linear combinations of orthogonal polynomials on the real line $\{{\mathcal{P}_{n}(x)}\}_{n\geq 1}$ as well as on the unit circle $\{\Phi_n(z)\}_{n\geq 1}$ has been a vibrant area of research. The exploration of linear combinations of two consecutive degrees of orthogonal polynomials on the real line (OPRL) was initially studied by Riesz \cite{Riesz} while proving the Hamburger moment problem. The necessary and sufficient condition for the orthogonality of linear combinations of OPRL has been investigated in \cite{When do linear}. For the Jacobi polynomials $P_n^{(\alpha,\beta
)}(x)$, the behaviour of zeros and interlacing properties of Jacobi polynomials and quasi-Jacobi polynomials are discussed in \cite{Driver_Jordaan_SIGMA_2016}.
In another direction, classical orthogonal polynomials including Jacobi polynomilas are studied in the perspective of exceptional orthogonal polynomials. For example, we refer to \cite{Simanek_Exceptional OP_CMFT_2023} and references therein.
The orthogonality of self-perturbations of  those orthogonal polynomials, which are generated by Christoffel transformation of a measure on the real line, has been studied in \cite{Vikas_Swami_quasi-type kernel}. In the same article, the recovery of the original orthogonal polynomials from quasi-type kernel polynomials and polynomials generated by linear spectral transformations are established. There is a close connection between the quadrature formula and linear combination of OPRL. More specifically, \cite{Peherstorfer_1990_MOC_quasi} provides the sufficient conditions on the coefficients $a_t$ for the polynomial $Q_n(x)=\sum_{t=0}^{k}a_t\mathcal{P}_{n-t}(x)$ to have $n$ distinct zeros within the interval of orthogonality. These conditions ensure that when these zeros are utilized as nodes in an interpolatory quadrature formula, the weights of the quadrature formula remain positive.

In \cite{Branquinho_Paco_1996_Generating}, the linear combination of two elements from a sequence of monic orthogonal polynomials on the unit circle (OPUC) are investigated. They introduced the expression
\begin{align*}
	\tilde{\Phi}_n(z)=\Phi_n(z)-a_n\Phi_{n-1}(z),
\end{align*}
and explored the necessary conditions on the parameters $a_n$ for which the sequence ${\tilde{\Phi}_n(z)}$ becomes orthogonal with respect to some specific measure. Later, in  \cite{Branquinho_Golinskii_Paco-CVandEE-1999}, the authors provided both necessary and sufficient conditions for the orthogonality of ${\tilde{\Phi}_n(z)}$. In the same paper, they also presented a description of the orthogonality measure associated with the new sequence $\{\tilde{\Phi}_n(z)\}$. More generally, the orthogonality of a finite linear combination of orthogonal polynomials on the unit circle with respect to the Bernstein-Szeg{\"o} measure is discussed in \cite{Paco_Peherstorfer_Orthogonality of FLOP_ACM_1996}.
 In \cite{Cacha_Paco_Proc. AMS_1998_CLC of OPS}, the necessary and sufficient conditions are established for a sequence of monic OPUC $\{\Psi_n\}$ such that the convex linear combination of $\{\Phi_n\}$ and $\{\Psi_n\}$ becomes orthogonal with respect to a specific measure.

The CD kernel of the orthonormal polynomials ${\phi_n}$ is defined as:
\begin{align}\label{CD kernel polynomial}
	\mathbb{K}_n(z,w,\mu)=\sum_{k=0}^{n}\overline{\phi_k(w)}\phi_k(z).
\end{align}
A more concise expression for this kernel is known as the Christoffel-Darboux formula \cite[eq. 3.2]{Brac_Finkel_Ranga_Veronese_2018_JAT}. For any $z,w\in \mathbb{C}$, with $z\bar{w}\neq1$, the Christoffel-Darboux kernel can be written as:
\begin{align}
	\mathbb{K}_n(z,w,\mu)=\frac{\overline{\phi_n^*(w)}\phi_n^*(z)-z\bar{w}\overline{\phi_n(w)}\phi_n(z)}{1-z\bar{w}}.
\end{align}
Next, we state the class of measures which we call ``Marcell\'{a}n class". The definition can be found in \cite{Branquinho_Golinskii_Paco-CVandEE-1999}.
\begin{definition}\label{Def. Quasi OPUC}
	 Let $\{\Phi_n\}$  be sequence of monic orthogonal polynomials with respect to a non-trivial positive Borel measure $\mu$ on the unit circle. If there exists a sequence of constants $\{a_n\}$, where $a_n\in\mathbb{C}$  and a measure $\tilde{\mu}$ on the unit circle such that the sequence of polynomials $\{\tilde{\Phi}_n\}$ defined by
	\begin{align}\label{Quasi OPUC eq.}
		\tilde{\Phi}_n(z)=\Phi_n(z)-a_n\Phi_{n-1}(z)
	\end{align}
	forms a monic orthogonal polynomial system with respect to $\tilde{\mu}$. Then we say that the measure $\mu$ belongs to the Marcell\'{a}n class denoted by $\mathcal{M}_2$.
\end{definition}
The primary aim of this manuscript is to explore the theory of quasi-orthogonal polynomials of order one on the unit circle, while also deriving estimates based on Verblunsky coefficients and the sequence of parameters $a_n$ for polynomials within the Marcellán class.
\subsection{Organization}
In Section \ref{sec:LC of OPUC}, given the orthogonal polynomial $\Phi_n$, we consider its quasi-orthogonal polynomial $\tilde{\Phi}_n$ from $\mathcal{M}_2$. We exhibited the expression which alternatively represent $\Phi_n$ as the linear combination of $\tilde{\Phi}_n$ and its reversed polynomial $\tilde{\Phi}^*_n$. This is achieved through a recurrence relation with appropriate variable coefficients. We conduct numerical experiments to observe the behavior of zeros of quasi-orthogonal polynomial of order one on the unit circle and discuss examples belonging to the Marcellán class. The discussion about obtaining the expression for the positive chain sequence in terms of Verblunsky coefficients from Definition \ref{Def. Quasi OPUC} is also presented.  In Section \ref{sec:Lubinsky type inequality}, we provide a detailed description of the Lubinsky inequality involving the measures $\mu$ and $\tilde{\mu}$, for which we prove certain norm inequalities. We define the  kernel-type polynomials and additional inequalities involving $\mu$ and $\tilde{\mu}$ are exhibited for kernel-type polynomials.

	\section{Orthogonal polynomials in the Marcell\'{a}n class}\label{sec:LC of OPUC}

It is clear from the Definition \ref{Def. Quasi OPUC} that $a_n\not\equiv0$, because $a_n\equiv0$ gives the equivalence of  $\tilde{\Phi}_n(z)$ and $\Phi_n(z)$ for each $n\in \mathbb{N}$. In fact, we can say more about the constant $a_n$'s.
The proof provided in \cite[Theorem 4]{Branquinho_Golinskii_Paco-CVandEE-1999}, which relies on the Szeg\"{o} recursion for the polynomial $\tilde{\Phi}_n(z)$, establishes that the constants $a_n\neq0$ hold for $n\geq1$.
The condition of non-zero $a_n$'s plays a crucial role in demonstrating that the polynomial $\Phi_{n}$ cannot be orthogonal to the constant function $1$ with respect to $\tilde{\mu}$.

\begin{proposition}\label{int PhiN non zero wrt tildemu}
	If $\mu$ is in the Marcell\'{a}n class $\mathcal{M}_2$ then, there does not exist any $N\in \mathbb{N}$ such that $\int\Phi_N(z)d\tilde{\mu}(z)=0$.
\end{proposition}
\begin{proof}
	Suppose, for contradiction, there exists an $N\in\mathbb{N}$ such that
	\begin{align*}
		\int\Phi_N(z)d\tilde{\mu}(z)=0.
	\end{align*}
Using \eqref{Quasi OPUC eq.} and the fact that $a_n\neq0$ for each $n\geq1$, we can deduce
\begin{align*}
		\int\Phi_{N-1}(z)d\tilde{\mu}(z)=0.
\end{align*}
By applying reverse induction, we establish
 \begin{align*}
 	\int\Phi_{n}(z)d\tilde{\mu}(z)=0 ~~ \text{for each}~~n=1,2,...,N.
 \end{align*}
This eventually leads to the conclusion that $\tilde{\mu}(\partial\mathbb{D})=0$, which is a contradiction.
\end{proof}

In \eqref{Quasi OPUC eq.}, we see that $\tilde{\Phi}_n$ is an nth degree monic polynomial written as a linear combination of known polynomials $\Phi_{n}$ and $\Phi_{n-1}$. Proposition \ref{OPUCintermsquasiopucanditsreverse}, dealt with expressing the orthogonal polynomial $\Phi_{n}$ in terms of  $ \tilde{\Phi}_{n+1}$ and its reverse polynomial, which is useful to give the information that $z=\bar{a}_{n+1}^{-1}$ will not be the zero of $\tilde{\Phi}_{n+1}^*(z)$ and $\Phi_n(z)$ for any $n\in \mathbb{N}$. \begin{proposition}\label{OPUCintermsquasiopucanditsreverse}
	For any positive Borel measure  $\mu$ in the Marcell\'{a}n class $\mathcal{M}_2$, we can have a sequence of constants $\{a_n\}$ such that the following recurrence relations
	\begin{align}\label{reverse recurrence quasiOPUC}
	((z-a_{n+1})(1-\bar{a}_{n+1}z)-\lvert\alpha_n\rvert^2z)\Phi_n(z)=(1-\bar{a}_{n+1}z)\tilde{\Phi}_{n+1}(z)+\bar{\alpha}_n\tilde{\Phi}_{n+1}^*(z)\\
	\nonumber ((z-a_{n+1})(1-\bar{a}_{n+1}z)-\lvert\alpha_n\rvert^2z)\Phi_n^*(z)=\alpha_nz\tilde{\Phi}_{n+1}(z)+(z-a_{n+1})\tilde{\Phi}_{n+1}^*(z)
	\end{align}
hold. Moreover, $a_{n+1}=\frac{\tilde{\Phi}_{n+1}(0)+\bar{\alpha}_n}{\bar{\alpha}_{n-1}}$ for non-zero $\alpha_{n}$'s.
\end{proposition}
\begin{proof}
	If $\mu \in \mathcal{M}_2$ then by Definition \ref{Def. Quasi OPUC}, there exists a complex sequence $\{a_n\}$ such that the polynomial sequence $\{\tilde{\Phi}_n\}$ is orthogonal with respect to $\tilde{\mu}$. Using the Szeg\"{o} recursion satisfied by $\Phi_{n}$, we can write the expression \eqref{Quasi OPUC eq.} as
	\begin{equation}\label{Szegorec.onquasiOPUC}
		\tilde{\Phi}_{n+1}(z)=(z-a_{n+1})\Phi_{n}(z)-\bar{\alpha}_n\Phi_{n}^*(z).
	\end{equation}
Applying the reverse operation * on both sides of  \eqref{Quasi OPUC eq.}, we get
\begin{equation}{\label{* oper.on Quasi OPUC}}
	\tilde{\Phi}_{n+1}^*(z)=\Phi_{n+1}^*(z)-\bar{a}_{n+1}z\Phi_{n}^*(z).
\end{equation}
Again we can use Szeg\"{o} recursion satisfied by $\Phi_{n}^*$ to write the expression \eqref{* oper.on Quasi OPUC} as
\begin{equation}{\label{Szegorec.on*oper.onquasiOPUC}}
	\tilde{\Phi}_{n+1}^*(z)=-\alpha_{n}z\Phi_{n}(z)+(1-\bar{a}_{n+1}z)\Phi_{n}^*(z).
\end{equation}
We can write \eqref{Szegorec.onquasiOPUC} and \eqref{Szegorec.on*oper.onquasiOPUC} in matrix form as follows:
\[\begin{pmatrix}
	\tilde{\Phi}_{n+1}(z)\\
	\tilde{\Phi}_{n+1}^*(z)
\end{pmatrix}=A(a_{n+1},\alpha_n;z)\begin{pmatrix}
	\Phi_{n}(z)\\
	\Phi_{n}^*(z)
\end{pmatrix},\]
with
\begin{align*}
	A(a_{n+1},\alpha_n;z)=\begin{pmatrix}
		z-a_{n+1} & -\bar{\alpha}_n\\
		-\alpha_nz & 1-\bar{a}_{n+1}z
	\end{pmatrix}.
\end{align*}
The matrix $A(a_{n+1},\alpha_n;z)$ is invertible since $\text{det}(A(a_{n+1},\alpha_n;z))=(z-a_{n+1})(1-\bar{a}_{n+1}z)-\lvert\alpha_{n}\rvert^2z\neq0.$
Hence we have
\[\begin{pmatrix}
\Phi_{n}(z)\\
	\Phi_{n}^*(z)
\end{pmatrix}=\frac{1}{(z-a_{n+1})(1-\bar{a}_{n+1}z)-\lvert\alpha_{n}\rvert^2z}\begin{pmatrix}
1-\bar{a}_{n+1}z & \bar{\alpha}_n\\
\alpha_nz & z-a_{n+1}
\end{pmatrix}\begin{pmatrix}
	\tilde{\Phi}_{n+1}(z)\\
	\tilde{\Phi}_{n+1}^*(z)
\end{pmatrix},\]
which gives the desired recurrence relations. If we substitute $z=0$ in \eqref{* oper.on Quasi OPUC}, then using \cite[Lemma 1.5.1]{SimonOPUC1}, we get $\tilde{\Phi}_{n+1}^*(0)=1$. Hence,  by plugging $z=0$ in \eqref{reverse recurrence quasiOPUC} and the fact that $\alpha_n=-\overline{\Phi_{n+1}(0)}$ gives $a_{n+1}=\frac{\tilde{\Phi}_{n+1}(0)+\bar{\alpha}_n}{\bar{\alpha}_{n-1}}$ for non-zero $\alpha_{n}$'s.
\end{proof}

\begin{corollary}
	If there does not exists any $n\in \mathbb{N}$ such that $\alpha_{n}=0$ then, $z=\bar{a}_{n+1}^{-1}$ is neither a zero of $\tilde{\Phi}_{n+1}^*(z)$ nor a zero  of $
	\Phi_{n}(z)$. Moreover, \begin{equation*}
		\alpha_n=-\frac{\tilde{\Phi}_{n+1}^*\left(\frac{1}{\bar{a}_{n+1}}\right)}{\Phi_{n}\left(\frac{1}{\bar{a}_{n+1}}\right)}\bar{a}_{n+1}.
	\end{equation*}
\end{corollary}
\begin{proof}
If we put $z=\bar{a}_{n+1}^{-1}$ in \eqref{reverse recurrence quasiOPUC}, then we have
\begin{equation*}
\bar{\alpha}_n\left(\Phi_{n}\left(\frac{1}{\bar{a}_{n+1}}\right)\frac{\alpha_{n}}{\bar{a}_{n+1}}+\tilde{\Phi}_{n+1}^*\left(\frac{1}{\bar{a}_{n+1}}\right)\right)=0.
\end{equation*}
If there does not exists any $n\in \mathbb{N}$ such that $\alpha_{n}=0$ then
\begin{equation*}
	\Phi_{n}\left(\frac{1}{\bar{a}_{n+1}}\right)\frac{\alpha_{n}}{\bar{a}_{n+1}}+\tilde{\Phi}_{n+1}^*\left(\frac{1}{\bar{a}_{n+1}}\right)=0.
\end{equation*}
Three cases arise:
\begin{enumerate}
	\item If $\Phi_{n}\left(\frac{1}{\bar{a}_{n+1}}\right)=0$ then, $\tilde{\Phi}_{n+1}^*\left(\frac{1}{\bar{a}_{n+1}}\right)=0$, which is a contradiction.
	\item If $\tilde{\Phi}_{n+1}^*\left(\frac{1}{\bar{a}_{n+1}}\right)=0$ then, $\Phi_{n}\left(\frac{1}{\bar{a}_{n+1}}\right)=0$, which is a contradiction.
	\item The only possibility is that $\Phi_{n}\left(\frac{1}{\bar{a}_{n+1}}\right)$ and $\tilde{\Phi}_{n+1}^*\left(\frac{1}{\bar{a}_{n+1}}\right)$ are non-zero. Hence
	\begin{equation*}
		\alpha_n=-\frac{\tilde{\Phi}_{n+1}^*\left(\frac{1}{\bar{a}_{n+1}}\right)}{\Phi_{n}\left(\frac{1}{\bar{a}_{n+1}}\right)}\bar{a}_{n+1}.
	\end{equation*}
\end{enumerate}
This completes the proof.
\end{proof}
\begin{example}\label{Lebesgue_M2}
	Let $d\mu(z)=\lvert dz\rvert$ be a Lebesgue measure on the unit circle and the sequence $\{\Phi_{n}\}$ with $\Phi_{n}(z)=z^n$ is an orthogonal polynomials sequence. Then we can find the non-zero complex constants $a_n\equiv a$ such that
	\begin{align}\label{Lebesgue measure not in M2}
		\tilde{\Phi}_n(z)=z^n-a z^{n-1},~~\text{for}~~ n\geq1,
	\end{align}
	constitutes an orthogonal polynomial system with respect to measure $d\tilde{\mu}(z)=\frac{\lvert dz\rvert}{\lvert z-a\rvert^2}$. Indeed,
	\begin{align*}
		\int_{\partial\mathbb{D}}\tilde{\Phi}_n(z)\overline{\tilde{\Phi}_m(z)}d\tilde\mu(z)=\int_{\partial\mathbb{D}}z^{n-m}\lvert z-a\rvert^2\frac{\lvert dz\rvert}{\lvert z-a\rvert^2}=\begin{cases}
			0, & n\neq m\geq1\\
			\neq 0, & n=m\geq1
		\end{cases}.
	\end{align*}
	Since $\{\tilde{\Phi}_n\}$ is a sequence of orthogonal polynomials on the unit circle, then by \rm\cite[Theorem 1.7.1]{SimonOPUC1}, $a_n\equiv a$ lies in the unit disc.
\end{example}
Note that for $a=1$ in \eqref{Lebesgue measure not in M2}, we can recover the polynomial $\Phi_{n}(z)=z^n$ of degree $n$ using \cite[equation 3.9]{Ismail_Ranga_GEVP_OPUC_Linear algebra_2019}.
The method illustrated in the preceding example yields a new measure $\tilde{\mu}$. In the subsequent example, we initiate with the measure constructed in the previous case and explore the measure associated with the linear combination of OPUC.
\begin{example}\label{B-S_M2}
	Let $d\mu(z)=\frac{\lvert dz\rvert}{\lvert z-a\rvert^2}$ be a positive measure on the unit circle and the sequence $\{\Phi_{n}\}$ defined by $\Phi_{n}(z)=z^{n-1}(z-a)$ for $n\geq1$ forms an orthogonal polynomials sequence with respect to $\mu$. Then there exist $a_{n}\equiv b\neq 0$ such that
	\begin{align}\label{first iteration of Lebesgue measure not in M2}
		\tilde{\Phi}_n(z)=z^{n-1}(z-a)-b z^{n-2}(z-a)
	\end{align}
	is an orthogonal polynomial sequence with respect to measure $d\tilde{\mu}(z)=\frac{\lvert dz\rvert}{\lvert z-b \rvert^2\lvert z-a\rvert^2}$. Indeed, for $n,m\geq2$, we have
	\begin{align*}
		\int_{\partial\mathbb{D}}\tilde{\Phi}_n(z)\overline{\tilde{\Phi}_m(z)}d\tilde\mu(z)=\int_{\partial\mathbb{D}}z^{n-m}\lvert z-b \rvert^2\lvert z-a\rvert^2\frac{\lvert dz\rvert}{\lvert z-b \rvert^2\lvert z-a\rvert^2}=\begin{cases}
			0, & n\neq m\\
			\neq 0, & n=m
		\end{cases}.
	\end{align*}
	If $b=-\bar{a}$ then,  $\tilde{\Phi}_n(z)=z^{n-2}(z-a)(z+\bar{a})$ for $n\geq2$ is an orthogonal polynomial sequence with respect to $d\tilde{\mu}(z)=\frac{\lvert dz\rvert}{\lvert z+\bar{a} \rvert^2\lvert z-a\rvert^2}$.
\end{example}
\begin{remark}
	It is worth emphasizing that the sequence of polynomials obtained by differentiating the orthogonal polynomials $\tilde{\Phi}_n$ is not orthogonal. This can be demonstrated through the following argument: the condition for a monic orthogonal polynomial sequence to have its derivative also be an orthogonal polynomial sequence is that the original orthogonal polynomials are generated from the normalized Lebesgue measure. The proof of this assertion can be found in \rm\cite[Theorem 3.3]{Paco_Maroni_OPUC derivate orthogonal NS_1991_Const. app.}
\end{remark}
If $d\mu$ represents the canonical Christoffel transformation \cite{Garza_CT for matrix measure_CMFT_2021} of the Borel measure on the unit circle, given by $d\mu = |z-\gamma|^2d\nu$, where $\gamma\in \mathbb{C}$, and if $\mathbb{K}_{n-1}{(\gamma,\gamma, \nu)}>0$ for $n\geq1$, then there exists a sequence of monic orthogonal polynomials with respect to $\mu$, denoted by $\Phi_n(z;\gamma)$, such that the following relation holds:
\begin{align*}
	\Phi_{n-1}(z;\gamma)=\frac{1}{z-\gamma}\left(\Psi_n(z)-\frac{\Psi_n(\gamma)}{\mathbb{K}_{n-1}{(\gamma,\gamma, \nu)}}\mathbb{K}_{n-1}{(z,\gamma, \nu)}\right),
\end{align*}
where $\Psi_{n}(z)$ represents monic orthogonal polynomials with respect to $\nu$, see \cite[proposition 2.4]{Paco_ST for HTM_JCAM_2007}.\\
Considering a case where $d\nu$ represents a normalized Lebesgue measure and $\gamma=1$, we observe that the Verblunsky coefficients associated with the measure $d\mu=|z-1|^2\frac{d\theta}{2\pi}$ are given by $\alpha_{n}=-\frac{1}{n+2}$ for $n\geq0$. The corresponding monic orthogonal polynomials $\Phi_{n}(z;1)$ can be expressed as:
\begin{align}\label{OP wrt CT measure with gamma 1}
	\Phi_{n-1}(z;1)=\frac{1}{z-1}\left(z^n-\frac{1}{n}\sum_{k=0}^{n-1}z^k\right).
\end{align}
Now we define the quasi-type kernel polynomial  of order one at $\gamma=1$ on the unit circle (see \cite{Vikas_Swami_quasi-type kernel} for real line) as:
\begin{align}\label{Quasi OP wrt CT measure with gamma 1}
\nonumber	\tilde{\Phi}_{n}(z;1,a_n):=\Phi_{n}(z;1)&-a_{n}\Phi_{n-1}(z;1)\\
	&=\frac{1}{z-1}\left(z^{n}(z-a_n)+\frac{a_n}{n}\sum_{k=0}^{n-1}z^k-\frac{1}{n+1}\sum_{k=0}^{n}z^k\right).
\end{align}
Indeed, the sequence ${\tilde{\Phi}_{n}(z;1,a_n)}$ may not necessarily form an orthogonal polynomial sequence for arbitrary parameters $a_n$. Consequently, we observe that for different values of $a_n$, the zeros of ${\tilde{\Phi}_{n}(z;1,a_n)}$ are located outside the unit disc.
\begin{table}[ht]
	\begin{center}
		\resizebox{!}{1.6cm}{\begin{tabular}{|c|c|c|c|c|}
				\hline
				\multicolumn{2}{|c|}{Zeros of $\Phi_n(z;1)$}&\multicolumn{3}{|c|}{Zeros of $\tilde{\Phi}_n(z;1,a_n)$}\\
				\hline
			 $n=5$ &$n=6$&$n=5$, $a_n=\frac{1}{n+1}-i$   &$n=6$, $a_n=-1.16$&$n=5, a_n=\frac{n}{n+1}$\\
				\hline
				0.294195-0.668367i&0.410684-0.639889i&0.303024-0.987019i&-1.07313&0.0\\
				\hline
				0.294195+0.668367i&0.410684+0.639889i&-0.113232-0.862069i&-0.941198&0.0\\
				\hline
				-0.375695-0.570175i&-0.205144-0.683797i&0.204197-0.629154i&0.348803-0.674876i&0.0\\
				\hline
					-0.375695+0.570175i&-0.205144+0.683797i&-0.617454-0.213182i&0.348803+0.674876i&0.0\\
				\hline
				-0.670332&-0.634112-0.287655i&-0.443202+0.4331161i&-0.350213-0.67387i&0.0\\
				\hline
				-&-0.634112+0.287655i&-&-0.350213+0.67387i&-\\
				\hline
		\end{tabular}}
		\captionof{table}{Zeros of  $\Phi_n(z;1)$ and $\tilde{\Phi}_n(z;1,a_n)$}
		\label{Zeros at gamma 1}
	\end{center}
\end{table}

\begin{figure}[!ht]
	\centering
	\begin{minipage}[c]{0.4\textwidth}
		\centering
		\includegraphics[scale=0.5]{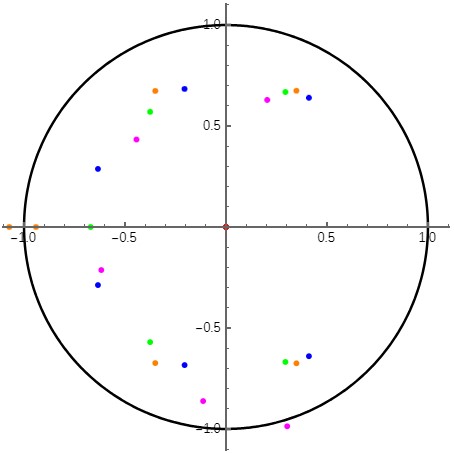}
		\caption{Zeros of $\Phi_5(z;1)$ (Green dots), $\Phi_6(z;1)$ (blue dots),$\tilde{\Phi}_5(z;1, \frac{1}{n+1}-i)$ (magenta dots),\\ $\tilde{\Phi}_5(z;1, 1.16)$ (orange dots) and $\tilde{\Phi}_5(z;1, \frac{n}{n+1})$ (red dots).}
		\label{Location of zeros of QuasiOPUC for 1}
	\end{minipage}
\end{figure}
Table \ref{Zeros at gamma 1} and Figure \ref{Location of zeros of QuasiOPUC for 1} illustrate that the zeros of $\Phi_{n}(z;1)$ are situated inside the unit disc, while at most one zero of $\tilde{\Phi}_n(z;1)$ lies outside the unit disc for various values of $a_n$, specifically when $a_n=\frac{1}{n+1}-i$ and $a_n=-1.16$. However, for $a_n=\frac{n}{n+1}$, the quasi-type kernel polynomial of order one at $\gamma=1$ becomes orthogonal with respect to the normalized Lebesgue measure, and consequently, all the zeros vanish. Therefore, there exists a sequence $a_n=\frac{n}{n+1}$ for which the Christoffel transformed measure of the normalized Lebesgue measure at $\gamma=1$ is contained within the Marcellán class.

When $\gamma=i$, the Verblunsky coefficients are given by $\alpha_n=\frac{(-1)^{n+2}i^{n+1}}{n+2}$, and the corresponding orthogonal polynomial is given by
\begin{align*}
	\Phi_{n-1}(z;i)=\frac{1}{z-i}\left(z^n-\frac{i^n}{n}\sum_{k=0}^{n-1}(-i)^kz^k\right).
\end{align*}
The quasi-type kernel polynomial at $\gamma=i$ of order one on the unit circle is defined as:
\begin{align*}
	\tilde{\Phi}_{n}(z;i)=\frac{1}{z-i}\left(z^{n}(z-a_n)+\frac{a_n i^n}{n}\sum_{k=0}^{n-1}(-i)^kz^k-\frac{i^{n+1}}{n+1}\sum_{k=0}^{n}(-i)^kz^k\right).
\end{align*}

\begin{table}[ht]
	\begin{center}
		\resizebox{!}{1.6cm}{\begin{tabular}{|c|c|c|c|c|}
				\hline
				\multicolumn{2}{|c|}{Zeros of $\Phi_n(z;i)$}&\multicolumn{3}{|c|}{Zeros of $\tilde{\Phi}_n(z;i,a_n)$}\\
				\hline
				$n=4$ &$n=5$&$n=4$, $a_n=\frac{n+1}{n}i$   &$n=5$, $a_n=1.1$&$n=4, a_n=\frac{n}{n+1}i$\\
				\hline
				0.67815+0.13783i&0.668367+0.294195i&1.0616i&1.02185-0.123134i&0.0\\
				\hline
				-0.67815+0.13783i&-0.668367+0.294195i&-0.480542-0.08298i&0.922406+0.178518i&0.0\\
				\hline
				-0.358285-0.53783i&-0.570175-0.375695i&0.480542-0.08298i&0.231037-0.633575i&0.0\\
				\hline
				0.358285-0.53783i&0.570175-0.375695i&-0.445624i&-0.638318+0.201751i&0.0\\
				\hline
				-&-0.670332i&-&-0.436978-0.456894i&-\\
				\hline
		\end{tabular}}
		\captionof{table}{Zeros of  $\Phi_n(z;i)$ and $\tilde{\Phi}_n(z;i,a_n)$}
		\label{Zeros at gamma i}
	\end{center}
\end{table}
\begin{figure}[!ht]
	\centering
	\begin{minipage}[c]{0.5\textwidth}
		\centering
		\includegraphics[scale=0.5]{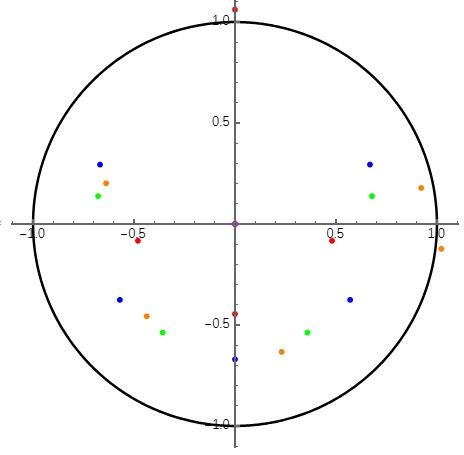}
		\caption{Zeros of $\Phi_4(z;i)$ (Green dots), $\Phi_5(z;i)$ (blue dots),$\tilde{\Phi}_4(z;i, \frac{n+1}{n}i)$ (red dots),\\ $\tilde{\Phi}_5(z;i, 1.1)$ (orange dots) and $\tilde{\Phi}_4(z;1, \frac{n}{n+1}i)$ (magenta dots).}
		\label{Location of zeros of QuasiOPUC for i}
	\end{minipage}
\end{figure}
Table \ref{Zeros at gamma i} and Figure \ref{Location of zeros of QuasiOPUC for i} depict that the zeros of $\Phi_{n}(z;i)$ are located within the unit disc. Note that, for $a_n=\frac{n+1}{n}i$ and $a_n=1.1$, at most one zero of $\tilde{\Phi}_n(z;i,a_n)$ falls outside the measure's support. However, when $a_n=\frac{n}{n+1}i$, the quasi-type kernel polynomial $\tilde{\Phi}_n(z;i,a_n)$ becomes orthogonal with respect to the normalized Lebesgue measure. Therefore, $d\mu=|z-i|^2\frac{d\theta}{2\pi}$ is encompassed within $\mathcal{M}_2$.

By generalizing Example \ref{Lebesgue_M2} and Example \ref{B-S_M2}, we arrive at the general form of measures in the Marcellán class, as outlined in \cite[Theorem 14]{Branquinho_Golinskii_Paco-CVandEE-1999}. This form consists of measures expressed as
\begin{align*}
	d\mu(z)=K\frac{|z-\bar{\beta} |^2}{|z-\chi_1|^2|z-\chi_2|^2} |dz|,
\end{align*}
where $0<|\beta|\leq1$, $\chi_1, \chi_2 \in \mathbb{D}$ and $K>0$. Such measures belong to $\mathcal{M}_2$. Additionally, the measure $d\tilde{\mu}$ is given by:
\begin{align*}
	d\tilde{\mu}(z)=K\frac{1}{|z-\chi_1 |^2|z-\chi_2|^2} |dz|.
\end{align*}



Assuming $\mu\in\mathcal{M}_2$, it follows from \cite[Theorem 4]{Branquinho_Golinskii_Paco-CVandEE-1999} that the sequence $\{\tilde{\Phi}_n\}$, where $\tilde{\Phi}_n(z)=z^{n-2} \tilde{\Phi}_2(z)$ for $n\geq2$, constitutes an orthogonal polynomial sequence with respect to the measure $\tilde{\mu}$. However, we currently lack information about the  polynomial $\tilde{\Phi}_2(z)$.
In the following proposition, we establish that it is possible to express the degree 2 polynomial $\tilde{\Phi}_2(z)$ in terms of lower degrees. We accomplish this by utilizing the Christoffel-Darboux formula.
\begin{proposition}
	If $\mu\in \mathcal{M}_2$ then, for $z\bar{w}\neq1$, we have
	\begin{align*}
		\frac{\overline{\tilde{\phi}^*_2(w)}\tilde{\phi}^*_2(z)-\overline{\tilde{\phi}_2(w)}\tilde{\phi}_2(z)}{1-z\bar{w}}	=\overline{\tilde{\phi_0}(w)}\tilde{\phi}_0(z)+\overline{\tilde{\phi}_1(w)}\tilde{\phi}_1(z).
	\end{align*}
	In particular, for $z=w$, we have
	\begin{align*}
		\lvert\tilde{\phi}^*_2(z)\rvert^2-\lvert\tilde{\phi}_2(z)\rvert^2=(1-\lvert z\rvert^2)\left(\lvert \tilde{\phi}_0(z)\rvert^2+\lvert \tilde{\phi}_1(z)\rvert^2 \right).
	\end{align*}
\end{proposition}
\begin{proof}
	For $n\geq2$, we have
	\begin{align}\label{Kerneloftildemu}
		\nonumber	\mathbb{K}_n(z,w,\tilde{\mu})&=\sum_{j=0}^{n}\overline{\tilde{\phi_j}(w)}\tilde{\phi}_j(z)=\overline{\tilde{\phi_0}(w)}\tilde{\phi}_0(z)+\overline{\tilde{\phi}_1(w)}\tilde{\phi}_1(z)+\sum_{j=2}^{n}\overline{\tilde{\phi_j}(w)}\tilde{\phi}_j(z)\\
		&=\overline{\tilde{\phi_0}(w)}\tilde{\phi}_0(z)+\overline{\tilde{\phi}_1(w)}\tilde{\phi}_1(z)+\overline{\tilde{\phi_2}(w)}\tilde{\phi}_2(z)\left(\frac{1-\bar{w}^{n-1}z^{n-1}}{1-\bar{w}z}\right).
	\end{align}
	We can write the right hand side of Christoffel-Darboux formula as
	\begin{align}\label{RHSofCDoftildemu}
		\frac{\overline{\tilde{\phi}^*_n(w)}\tilde{\phi}^*_n(z)-z\bar{w}\overline{\tilde{\phi}_n(w)}\tilde{\phi}_n(z)}{1-z\bar{w}}=\frac{\overline{\tilde{\phi}^*_2(w)}\tilde{\phi}^*_2(z)-z^{n-1}\bar{w}^{n-1}\overline{\tilde{\phi}_2(w)}\tilde{\phi}_2(z)}{1-z\bar{w}}.
	\end{align}
	By equating \eqref{Kerneloftildemu} and \eqref{RHSofCDoftildemu}  we get the desired result.
\end{proof}
In the following proposition, we establish the absolute continuity of the measure associated with the sequence of orthogonal polynomials.
\begin{proposition}\label{mu tilde is ac}
	If $\{\tilde{\Phi}_n\}$  be a sequence of orthogonal polynomials with respect to $\tilde{\mu}$ as given in Definition \ref{Def. Quasi OPUC}. Then $\tilde{\mu}$ is absolutely continuous with respect to Lebesgue measure and $\tilde{\mu}'(\theta)\in L^2(0, 2\pi)$.
\end{proposition}
\begin{proof}
	For $n\geq2$, we have
	\begin{align*}
		\lVert \tilde{\Phi}_n\rVert^2_{\tilde{\mu}}= \int_{\partial\mathbb{D}} 	\lvert \tilde{\Phi}_2(z)\rvert^2 d\tilde{\mu}(z)=C< \infty.
	\end{align*}
Since $\tilde{\Phi}^*_n(z)=\tilde{\Phi}^*_2(z)$, one can obtain
\begin{align*}
\frac{1}{2\pi}\int_{0}^{2\pi} \frac{\lVert \tilde{\Phi}_n\rVert^2_{\tilde{\mu}}}{\lvert\tilde{\Phi}^*_n(e^{i\theta})\rvert^4}d\theta=	\frac{1}{2\pi}\int_{0}^{2\pi} \frac{C}{\lvert\tilde{\Phi}^*_2(e^{i\theta})\rvert^4}d\theta\leq K~~\forall~~n.
\end{align*}
Hence, using \cite[Theorem 13.5]{Geronimus-OPUC and their applications-transl. (1962)}, $\tilde{\mu}$ is absolutely continuous with respect to Lebesgue measure and $\tilde{\mu}'(\theta)\in L^2(0, 2\pi)$.
\end{proof}
\subsection{Para-orthogonal polynomials and Chain Sequences}Gauss hypergeometric functions are pivotal in the study of orthogonal polynomials, as highlighted in \cite{Chihara book}. Furthermore, these hypergeometric functions serve as essential tools in diverse areas, including the exploration of Ramanujan's modular equations, as discussed in \cite{Qiu_Ma_Chu_CMFT_2022}.  The connection between chain sequences and the continued fraction representation of the ratio of Gauss hypergeometric functions is a well-known phenomenon. The integral representation of this ratio can be explored in \cite{Behera_Swami_Missing term_CMFT_2018}, while the linkage between chain sequences and orthogonal polynomials is elaborated in \cite{Chihara book}. Subsequently, we proceed to represent the positive chain sequence in relation to the Verblunsky coefficients ${\tilde{\alpha}_n}$ associated with the measure $\tilde{\mu}$. To achieve this, it is essential to introduce the notion of para-orthogonal polynomials on the unit circle (POPUC).
The POPUC associated with $\tilde{\Phi}_n$ is given by
\begin{align}\label{Para of tildephi}
	\tilde{\Phi}^p_n(z;\zeta):= \tilde{\Phi}_n(z)-\frac{\tilde{\Phi}_n(\zeta)}{\tilde{\Phi}^*_n(\zeta)} \tilde{\Phi}^*_n(z) ~~~\text{for}~~\zeta\in\partial\mathbb{D}.
\end{align}
The CD kernel can be expressed in the framework of the POPUC and this equivalence can be articulated as follows:
\begin{align*}
	\mathbb{K}_n(z,w,\tilde{\mu})=\frac{\overline{\tilde{\Phi}_{n+1}(\zeta)}k_{n+1}^2}{\bar{w}(z-w)}\tilde{\Phi}^p_{n+1}(z;\zeta).
	\end{align*}
The relationship between the POPUC and the CD kernel is extensively explored in \cite{Njastad_1996_CD kernel and para-OPUC}. This connection is instrumental in deriving further insights, as exemplified in \cite{CMV_Para OPUC_JAT_2002}. Notably, this linkage plays a crucial role in understanding the distribution of zeros of para-orthogonal polynomials, as demonstrated in \cite[section 2.14]{B.SimonSzegoDescendants}.

By utilizing \eqref{Para of tildephi}, we consider the sequence $\{\mathcal{L}_n(z;\zeta)\}$ of monic polynomials with deg $\mathcal{L}_n=n$ in the variable $z$ defined as
\begin{align}\label{Monic para of one degree less}
	\mathcal{L}_{n}(z;\zeta)=\frac{1}{1+\frac{\tilde{\Phi}_{n+1}(\zeta)}{\tilde{\Phi}^*_{n+1}(\zeta)}\tilde{\alpha}_{n+1}}\frac{\tilde{\Phi}^p_{n+1}(z;\zeta)}{z-\zeta},
\end{align}
where $\tilde{\alpha}_{n}$'s are Verblunsky coefficients associated with measure $\tilde{\mu}$.  An equivalent formulation of \eqref{Monic para of one degree less} is given by
\begin{align}\label{Monic para of one degree less_equiv form}
	\mathcal{L}_{n}(z;\zeta)=\frac{z\tilde{\Phi}_n(z)-\zeta\frac{\tilde{\Phi}_n(\zeta)}{\tilde{\Phi}^*_n(\zeta)}\tilde{\Phi}^*_n(z)}{z-\zeta}.
\end{align}
Now we have a sequence of polynomials $\{\mathcal{R}_n\}$, which is defined by
\begin{align}\label{poly. OPUC with chain sequence}
	\mathcal{R}_n(z)= \mathcal{T}_{n-1} \mathcal{L}_{n}(z;1),
\end{align}
where
\begin{align*}
	\mathcal{T}_{n-1}=	\frac{\prod\limits_{j=0}^{n-1}\left(1-\frac{\tilde{\Phi}_n(1)}{\tilde{\Phi}^*_n(1)}\tilde{\alpha}_j\right)}{\prod\limits_{j=0}^{n-1}\left(1-Re\left(\frac{\tilde{\Phi}_n(1)}{\tilde{\Phi}^*_n(1)}\tilde{\alpha}_j\right)\right)}.
\end{align*}
Thus, as established in \cite[Theorem 2.2]{Costa_Felix_Ranga_OPUC and chain seq_JAT_2013}, the sequence of polynomials (modified CD kernel) $\{\mathcal{R}_n\}$ satisfies the three-term recurrence relation. Notably, one of the recurrence coefficients in this sequence forms a positive chain sequence. More specifically, the recursive expression for $\{\mathcal{R}_n\}$ is given by
\begin{align}\label{TTRR of modified CD kernel which involves chain sequence}
	\mathcal{R}_{n+1}(z)=[(1+it_{n+1})z+(1-it_{n+1})]\mathcal{R}_{n}(z)-4c_{n+1}z\mathcal{R}_{n-1}(z),
\end{align}
with initial condition $\mathcal{R}_{-1}(z)=0$ and $\mathcal{R}_0(z)=1$ (see also \cite{Shukla_Swami_Malaysian_2023}). Significantly, the sequences $\{t_n\}$ represent real parameters, while $\{c_n\}$ constitutes a positive chain sequence. These parameters are determined by the following expressions:

\begin{align*}
	t_n=\begin{cases}
		-\frac{Im(\alpha_0)}{1-Re(\alpha_0)} & \text{for}~~n=1,\\
		-\frac{Im\left(\frac{1-\bar{\alpha}_0-a_1}{1-\alpha-\bar{a}_1}(\alpha_1-\bar{a}_2\alpha_0)\right)}{1-Re\left(\frac{1-\bar{\alpha}_0-a_1}{1-\alpha_0-\bar{a}_1}(\alpha_1-\bar{a}_2\alpha_0)\right)} & \text{for}~~n=2,\\
		0 & \text{for}~~n\geq3,
	\end{cases}
\end{align*}
and
\begin{align}\label{Chain sequence of M2 class}
	c_{n+1}=\begin{cases}
		\frac{1}{4}\frac{\left(1-\lvert\alpha_0-\bar{a}_1\rvert^2\right)\left|1-\frac{1-\bar{\alpha}_0-a_1}{1-\alpha-\bar{a}_1}(\alpha_1-\bar{a}_2\alpha_0)\right|^2}{\left(1-Re(\alpha_0-\bar{a}_1)\right)\left[1-Re\left(\frac{1-\bar{\alpha}_0-a_1}{1-\alpha_0-\bar{a}_1}(\alpha_1-\bar{a}_2\alpha_0)\right)\right]} & \text{for}~~n=1,\\
		\frac{1}{4}\frac{1-\lvert\alpha_1-\bar{a}_2\alpha_0\rvert^2}{\left[1-Re\left(\frac{1-\bar{\alpha}_0-a_1}{1-\alpha_0-\bar{a}_1}(\alpha_1-\bar{a}_2\alpha_0)\right)\right]} & \text{for}~~n=2,\\
		\frac{1}{4} & \text{for}~~n\geq3.
	\end{cases}
\end{align}
The positive chain sequence $\{c_n\}_{n\geq1}$ can be expressed as $c_n=(1-g_{n-1})g_n$ for $n\geq1$, wherein the sequence $\{g_n\}_{n\geq0}$ is referred to as the parameter sequence associated with $c_n$ \cite{Behera_Ranga_Swami_SIGMA_2016}. This parameter sequence is given by
\begin{align*}
	g_n=\begin{cases}
		\frac{1}{2}\frac{\lvert1-\alpha_0\rvert^2}{1-Re(\alpha_0)} & \text{for}~~ n=0,\\
		\frac{1}{2}\frac{\left|1-\frac{1-\bar{\alpha}_0-a_1}{1-\alpha-\bar{a}_1}(\alpha_1-\bar{a}_2\alpha_0)\right|^2}{\left[1-Re\left(\frac{1-\bar{\alpha}_0-a_1}{1-\alpha_0-\bar{a}_1}(\alpha_1-\bar{a}_2\alpha_0)\right)\right]} & \text{for}~~n=1,\\
		\frac{1}{2} & \text{for}~~n\geq 2.
	\end{cases}
\end{align*}
In addition,  for $n\geq2$ and $\zeta=1$, we can write \eqref{Monic para of one degree less_equiv form} as
\begin{align*}
	\mathcal{L}_{n}(z;1)=\frac{z^{n-1}\tilde{\Phi}_2(z)-\frac{\tilde{\Phi}_2(1)}{\tilde{\Phi}^*_2(1)}\tilde{\Phi}^*_2(z)}{z-1}.
\end{align*}
For $\lvert z\rvert<1$, we have
\begin{align*}
	\tilde{\Phi}^*_2(z)=\lim\limits_{n\rightarrow\infty}\frac{(1-z)\tilde{\Phi}^*_2(1)}{\tilde{\Phi}_2(1)}\mathcal{L}_{n}(z;1).
\end{align*}
Thus,
\begin{align*}
	\lim\limits_{n\rightarrow\infty}\mathcal{R}_n(z)=\frac{1-\alpha_0}{1-Re(\alpha_0)}\frac{\left(1-\frac{1-\bar{\alpha}_0-a_1}{1-\alpha-\bar{a}_1}(\alpha_1-\bar{a}_2\alpha_0)\right)}{\left[1-Re\left(\frac{1-\bar{\alpha}_0-a_1}{1-\alpha_0-\bar{a}_1}(\alpha_1-\bar{a}_2\alpha_0)\right)\right]}\frac{\tilde{\Phi}_2(1)}{(1-z)\tilde{\Phi}^*_2(1)}\tilde{\Phi}^*_2(z).
\end{align*}
Note that for $t_n=0$ and $c_n=\frac{1}{4}$, \eqref{TTRR of modified CD kernel which involves chain sequence} gives
\begin{align}
	\mathcal{R}_{n+1}(z) = (z + 1)\mathcal{R}_{n}(z) - z\mathcal{R}_{n-1}(z).
\end{align}
The polynomial corresponding to this recurrence relation is given by
\begin{align}\label{Rational_Mod_POPUC_1}
	\mathcal{R}_n(z) = \frac{z^{n+1} - 1}{z - 1}.
\end{align}
The monic Christoffel polynomial of degree $n$ defined in \eqref{OP wrt CT measure with gamma 1} can be written in terms of the derivative of $\mathcal{R}_n(z)$ as:
\begin{align}
	\Phi_n(z;1) = \frac{1}{n+1} \frac{d}{dz}(z\mathcal{R}_{n}(z)) = \frac{(n+1)z^n + nz^{n-1} + \ldots + 2z + 1}{n+1}.
\end{align}
Additionally, the orthogonal polynomial corresponding to the Lebesgue measure on the unit circle can be expressed in terms of $\mathcal{R}_n(z)$ as follows:
\begin{align*}
	\Phi_n(z) = z^n = \mathcal{R}_{n}(z) - \mathcal{R}_{n-1}(z).
\end{align*}
 The corresponding POPUC is given by
\begin{align*}
	\Phi^p_n(z; \zeta) = z^n - \zeta^n \quad \text{for} \quad \zeta \in \partial \mathbb{D}.
\end{align*}
The polynomial $\Phi^p_n(z; \zeta)$, which has degree $n$, is orthogonal to the set $\{z, z^2, \ldots, z^{n-1}\}$ with respect to the Lebesgue measure, and all the zeros of $\Phi^p_n(z; \zeta)$ lie on the unit circle. Specifically, for $\zeta = 1$, we have
\begin{align}\label{POPUC Lebesgue for zeta1}
	\Phi^p_n(z; 1) = z^n - 1.
\end{align}
Upon taking the linear combination of $\Phi^p_n(z;1)$ and $\Phi^p_{n-1}(z;1)$, we obtain
\begin{align}\label{POPUC LC gamman}
	\Phi^p_n(z;1;\gamma_n) := \Phi^p_n(z;1) - \gamma_n \Phi^p_{n-1}(z;1) = z^n - 1 - \gamma_n (z^{n-1} - 1).
\end{align}
The polynomial $\Phi^p_n(z;1;\gamma_n)$, of degree $n$, is orthogonal to the set $\{z, z^2, \ldots, z^{n-2}\}$ with respect to the Lebesgue measure. It may be noted that, we are not calling \eqref{POPUC LC gamman} as quasi POPUC, because in the literature \cite{Bultheel_Quasi_POPUC_JCAM_2022}, this terminology is used for another sequence of polynomials. Interestingly, the quasi POPUC of order one given in \cite{Bultheel_Quasi_POPUC_JCAM_2022} coincides with the POPUC.
This same article also discusses higher orders of quasi POPUC and the properties of their zeros.

Interestingly, no longer, all the zeros of $\Phi^p_n(z;1;\gamma_n)$ lie on the unit circle. Particularly, when $\gamma_n = \gamma\in (0, 1)$, \eqref{POPUC LC gamman} becomes
\begin{align}\label{POPUC LC gammain01}
	\Phi^p_n(z;1;\gamma) := z^n - \gamma z^{n-1} + \gamma - 1.
\end{align}
It is evident that the zeros of \eqref{POPUC Lebesgue for zeta1} lie on the unit circle. However, when we perturb the polynomial to obtain \eqref{POPUC LC gamman}, with one zero at $z=1$ and the others depending on the parameters $\gamma_n$, interesting patterns emerge. For instance, setting $\gamma_n=\gamma\in(0,1)$, the zeros of \eqref{POPUC LC gammain01}, except for $z=1$, reside within the unit disk. This behavior is illustrated for $n=6$, $t=0.9$ and $n=7$, $t=0.2$ in the  Table \ref{Table_Zeros_of_POPUC_LCPOPUC} and figures \ref{Zeros_LC_POPUC_app._origin} and \ref{Zeros_LC_POPUC_app._circle}.


 On the other hand, if $\gamma>1$, at most one zero of \eqref{POPUC LC gamman} extends beyond the unit disk, as detailed in the Table \ref{Table_Zeros_of_POPUC_LCPOPUC}. Additionally, as $\gamma$ approaches 1 from within $(0,1)$, the zeros of \eqref{POPUC LC gammain01} tend towards the origin, while as $\gamma$ approaches 0, they move closer to the unit circle. The rational modification of \eqref{POPUC Lebesgue for zeta1} given by \eqref{Rational_Mod_POPUC_1} exhibits a distinct behavior, with all its zeros lying on the boundary of the unit disk. Furthermore, figures \ref{Zeros_LC_POPUC_app._origin} and \ref{Zeros_LC_POPUC_app._circle} depict an intriguing alternation in the placement of zeros between \eqref{Rational_Mod_POPUC_1} and \eqref{POPUC Lebesgue for zeta1} along the unit circle.

\begin{table}[ht]
	\begin{center}
		\resizebox{!}{1.6cm}{\begin{tabular}{|c|c|c|c|}
				\hline
				\multicolumn{4}{|c|}{Zeros of $\Phi^p_n(z;1;\gamma)$}\\
				\hline
				$n=6$, $t=0.9$ &$n=7$, $t=0.2$ &$n=6$, $t=2$   &$n=7$, $t=9.1$\\
				\hline
				-0.583128&-0.846258-0.419346i&-0.678351-0.458536i&-0.964456\\
				\hline
				-0.237066-0.553379i&-0.846258+0.419346i&-0.678351+0.458536i&-0.49786-0.834243i\\
				\hline
					-0.237066+0.553379i&-0.187738-0.942i&0.195377-0.848854i&-0.49786+0.834243i\\
				\hline
				0.47863-0.494043i&-0.187738+0.942i&0.195377+0.848854i&0.480095-0.864496i\\
				\hline
				0.47863+0.494043i&0.633997-0.755073i&1&0.480095+0.864496i\\
				\hline
				1&0.633997+0.755073i&1.96595&1\\
				\hline
				-&1&-&9.0999\\
				\hline
		\end{tabular}}
		\captionof{table}{Zeros of $\Phi^p_n(z;1;\gamma)$}
		\label{Table_Zeros_of_POPUC_LCPOPUC}
	\end{center}
\end{table}
\begin{minipage}{1.0\linewidth}
	\centering
	\begin{minipage}{0.44\linewidth}
		\begin{figure}[H]
			\includegraphics[width=\linewidth]{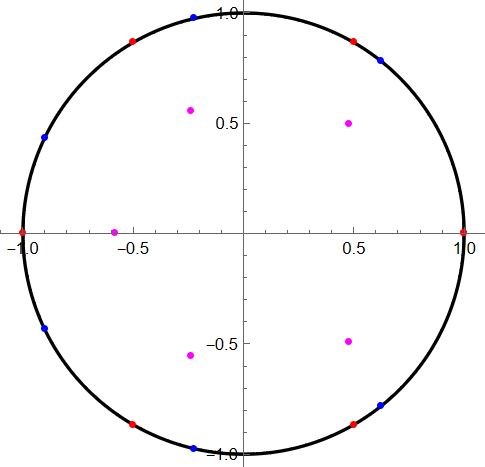}
			\captionof{figure}{Zeros of $\Phi^p_6(z;1;0.9)$ (Magenta dots), $\Phi^p_6(z;1)$ (Red dots), $\mathcal{R}_6(z)$ (Blue dots)}
			\label{Zeros_LC_POPUC_app._origin}
		\end{figure}
	\end{minipage}
	\hspace{0.08\linewidth}
	\begin{minipage}{0.44\linewidth}
		\begin{figure}[H]
			\includegraphics[width=\linewidth]{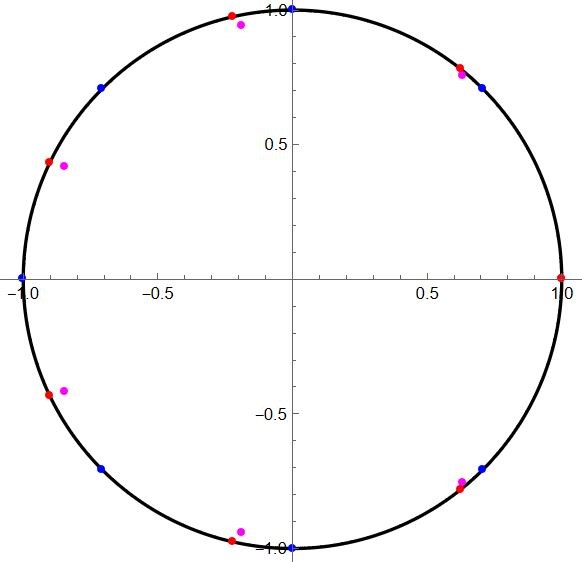}
			\captionof{figure}{Zeros of $\Phi^p_7(z;1;0.2)$ (Magenta dots), $\Phi^p_7(z;1)$ (Red dots), $\mathcal{R}_7(z)$ (Blue dots)}
			\label{Zeros_LC_POPUC_app._circle}
		\end{figure}
	\end{minipage}
\end{minipage}
\section{Lubinsky type inequalities involving \texorpdfstring{$\mu$}{} and \texorpdfstring{$\tilde{\mu}$}{}} \label{sec:Lubinsky type inequality}
In the previous section, it has been observed that the zeros of the linear combination of POPUC and quasi-orthogonal polynomials may lie outside the support of the measure. This leads to the question of examining the relation between the measure $\mu$ in the Marcell\"{a}n class and its corresponding measure $\tilde{\mu}$ involving the quasi orthogonal polynomials. In this section, we determine several inequalities involving $\mu$ and $\tilde{\mu}$. In particular, we are interesting in finding an inequality similar to Lubinsky inequality \cite{Lubinsky2007annals}, which require one of the following norm inequalities.
\subsection{Norm inequalities}\label{subsec: Involving norm inequalities}
Here, we discuss some norm inequalities involving $\tilde{\Phi}_n$ and $\Phi_{n}$ with respect to different measures.
\begin{theorem}
	Let $\{\Phi_{n}\}$ be a sequence of orthogonal polynomials on the unit circle with respect to the measure $\mu$. The following inequalities hold:
	\begin{enumerate}
		\rm\item Let $\tilde{\Phi}_{n}(z)=\Phi_{n}(z)-a_n\Phi_{n-1}(z)$ be a polynomial of degree exactly equal to n. Then
		\begin{align}\label{quasi OP norm wrt mu bound}
			\lVert\tilde{\Phi}_n\rVert_\mu^2\leq(1+\lvert a_n\rvert^2)m_o(\mu),
		\end{align}
		where $m_o(\mu)=\int_{0}^{2\pi}d\mu(e^{i\theta})$.
		\item If $\mu\in\mathcal{M}_2$ then, $\lVert\Phi_{n}\rVert_{\tilde{\mu}}\geq\tilde{\mu}(\partial\mathbb{D})\prod\limits_{j=1}^{n}\lvert a_j\rvert$.
		\item  If $\mu\in\mathcal{M}_2$ then,  $\lVert\tilde{\Phi}_{n}\rVert_{\tilde{\mu}}^2\leq2(1+\lvert a_n\rvert)^4\lVert\Phi_{n}\rVert_{\tilde{\mu}}^2$.
	\end{enumerate}
\end{theorem}
\begin{proof} Let $\tilde{\Phi}_{n}(z)=\Phi_{n}(z)-a_n\Phi_{n-1}(z)$ be a polynomial of degree n.
	\begin{enumerate}
		\item Using the minimization property of $\Phi_n$ with respect to $\mu$, we get \begin{align*}
			\lVert\tilde{\Phi}_n\rVert_\mu^2&=\int\lvert\tilde{\Phi}_{n}(z)\rvert^2d\mu(z)\\
			&=\int\left(\Phi_{n}(z)-a_n\Phi_{n-1}(z)\right)\left(\overline{\Phi_{n}(z)}-\overline{a_n}\overline{\Phi_{n-1}(z)}\right)d\mu(z)\\
			&=	\lVert\Phi_n\rVert_\mu^2+\lvert a_n\rvert^2\lVert\Phi_{n-1}\rVert_\mu^2\\
			&\leq \lVert z^n\rVert_\mu^2+\lvert a_n\rvert^2\lVert z^{n-1}\rVert_\mu^2\\
			&=(1+\lvert a_n\rvert^2)m_o(\mu),
		\end{align*}	where $m_o(\mu)=\int_{0}^{2\pi}d\mu(e^{i\theta})$.
		\item  If $\mu\in\mathcal{M}_2$, then there exists $a_n\not\equiv0$ such that $\tilde{\Phi}_n$'s are orthogonal with respect to $\tilde{\mu}$. This gives
		\begin{align}\label{bound.an.interms.Phi_tilde}
			\nonumber		& \langle\tilde{\Phi}_n, \Phi_{n-1} \rangle_{\tilde{\mu}}=	\int\tilde{\Phi}_n(z)\overline{\Phi_{n-1}(z)}d\tilde{\mu}(z)=0\\
			\nonumber		\implies &a_n=-\frac{\langle\Phi_{n}, \Phi_{n-1}\rangle_{\tilde{\mu}}}{\lVert\Phi_{n-1}\rVert_{\tilde{\mu}}^2}\\
			\implies &\lvert a_n\rvert\leq\frac{\lVert\Phi_{n}\rVert_{\tilde{\mu}}}{\lVert\Phi_{n-1}\rVert_{\tilde{\mu}}}.
		\end{align}
		Recursively, we obtain
		\begin{align*}
			\frac{1}{\lVert\Phi_{n}\rVert_{\tilde{\mu}}}\leq\frac{1}{\tilde{\mu}(\partial\mathbb{D})\prod\limits_{j=1}^{n}\lvert a_j\rvert}.
		\end{align*}
		\item If $\mu\in\mathcal{M}_2$ then, there exists $a_n\not\equiv0$ such that $\tilde{\Phi}_n$'s are orthogonal with respect to $\tilde{\mu}$. Then
		\begin{align*}
			\lVert\tilde{\Phi}_n\rVert_{\tilde{\mu}}^2&=\int\lvert\tilde{\Phi}_{n}(z)\rvert^2d\tilde{\mu}(z)\\
			&=\int\left(\Phi_{n}(z)-a_n\Phi_{n-1}(z)\right)\left(\overline{\Phi_{n}(z)}-\overline{a_n}\overline{\Phi_{n-1}(z)}\right)d\tilde{\mu}(z)\\
			&=	\lVert\Phi_n\rVert_{\tilde{\mu}^2}^2+\lvert a_n\rvert^2\lVert\Phi_{n-1}\rVert_{\tilde{\mu}^2}^2-2\int Re(\overline{a_n}\overline{\Phi_{n-1}(z)}\Phi_{n}(z))d\tilde{\mu}(z).
		\end{align*}
		This implies
		\begin{flalign*}
			\big|\lVert\tilde{\Phi}_n\rVert_{\tilde{\mu}}^2-\lVert\Phi_n\rVert_{\tilde{\mu}}^2\big| &\leq 2\lvert a_n\rvert^2\lVert\Phi_{n-1}\rVert_{\tilde{\mu}}^2+\lVert\Phi_{n}\rVert_{\tilde{\mu}}^2\\
			&\leq \lVert\Phi_{n}\rVert_{\tilde{\mu}}^2\left(1+2\lvert a_n\rvert^4\right)\hspace{1cm}\text{ by using \eqref{bound.an.interms.Phi_tilde}}.
		\end{flalign*}
			\end{enumerate}
			This completes the proof.
\end{proof}
\begin{example}
We observe that if $a_n=\frac{n}{n+1}$, $n\geq1$ then $d\mu=|z-1|^2\frac{d\theta}{2\pi}\in \mathcal{M}_2$, and the orthogonal polynomials with respect to $d\mu$ are given by \eqref{OP wrt CT measure with gamma 1}. Meanwhile, the quasi-orthogonal polynomial of order one is given by \eqref{Quasi OP wrt CT measure with gamma 1}. Since
	\begin{align*}
		m_o(d\mu)=\int_{0}^{2\pi}|e^{i\theta}-1|^2\frac{d\theta}{2\pi}=2,
	\end{align*}
	we have $\lVert\tilde{\Phi}_n\rVert_\mu^2\leq 2(1+|a_n|^2)$. When $a_n=\frac{n}{n+1}$, this yields $\lVert\tilde{\Phi}_n\rVert_\mu^2\leq 2\left(\frac{1+2(n+1)n}{(n+1)^2}\right)$. Additionally, for $d\tilde{\mu}(\theta)=\frac{d\theta}{2\pi}$, the lower bound of the norm of $\Phi_{n}$ becomes
	\begin{align*}
		\lVert\Phi_{n}\rVert_{\tilde{\mu}}\geq\frac{1}{n+1}.
	\end{align*}
\end{example}
\subsection{Lubinsky type inequality without ordering of the measures}
In \cite{Lubinsky2007annals}, Lubinsky introduced an important tool, which Simon calls Lubinsky inequality \cite{B.SimonSzegoDescendants}, to prove the universality limits, which is essential for giving information about the zeros of para-orthogonal polynomials on the unit circle (POPUC). We have a well-developed theory in the literature \cite{Nevai_Case study Christoffel functions_1986,SimonOPUC1} to study the asymptotic behaviour of the Christoffel function, which makes the Lubinsky inequality ``a powerful tool" since this allows us to control the off-diagonal CD kernel to diagonal CD kernel. If we drop the hypothesis of the Lubinsky inequality (i.e. work with a general pair of measures), then the upper bound for $L^2$-norm of the CD kernel of $\mu_2$ with respect to $\mu_1$ will create a challenge. In our setting of Theorem \ref{LubinskytypeinequalityforM2class}, we drop the hypothesis of the Lubinsky inequality and get the control of the off-diagonal CD kernel from the diagonal CD kernel and the summation involving $a_i$'s.

For clarity, we give the statement of the Lubinsky inequality. The proof of the  same can be found in \cite{B.SimonCDkernel2008} (see also \cite{Lubinsky2007annals}).
\begin{theorem}\label{Lubinsky Inequality}
	Let $\mu_1\leq \mu_2$. Then for any complex constants $z$ and $w$, we have
	\begin{flalign*}
		\lvert\mathbb{K}_n(z,w,\mu_1)-\mathbb{K}_n(z,w,\mu_2)\rvert^2
		\leq\mathbb{K}_n(w,w,\mu_1)\left[\mathbb{K}_n(z,z,\mu_1)-\mathbb{K}_n(z,z,\mu_2)\right].
	\end{flalign*}
\end{theorem}
\begin{lemma}
	Let $\mu \in \mathcal{M}_2$ be a positive Borel measure on the unit circle as defined in Definition \ref{Def. Quasi OPUC}  and $\mathbb{K}_n(z,s,\tilde{\mu})$ denote the CD kernel corresponding to the measure $\tilde{\mu}$ as defined in \eqref{CD kernel polynomial}. Then
	\begin{align}\label{L2norm.ineq.Kernel.of.quasi}
	\lVert\mathbb{K}_n(z,s,\tilde{\mu})\rVert_{\mu}^2\leq m_0(\mu)\mathbb{K}_n(z,z,\tilde{\mu}) \sum_{j=0}^{n}(1+\lvert a_{j}\rvert^2)\tilde{k}_j^{2},
	\end{align}
where $\tilde{k}_n=\lVert\tilde{\Phi}_{n}\rVert_{\tilde{\mu}}^{-1}$ and $m_o(\mu)=\int_{0}^{2\pi}d\mu(e^{i\theta})$.
\end{lemma}
\begin{proof} By the Cauchy-Schwarz inequality, we have
	\newline $\displaystyle \lvert\mathbb{K}_n(z,s,\tilde{\mu})\rvert^2\leq\left(\sum_{j=0}^{n}\lvert\tilde{\phi}_j(z)\rvert^2\right)\left(\sum_{j=0}^{n}\lvert\tilde{\phi}_j(s)\rvert^2\right),$\\
	which implies
	\newline $\displaystyle  \int\lvert\mathbb{K}_n(z,s,\tilde{\mu})\rvert^2d\mu(s)\leq\mathbb{K}_n(z,z,\tilde{\mu})\int\mathbb{K}_n(s,s,\tilde{\mu})d\mu(s)$.\\
	Using \eqref{quasi OP norm wrt mu bound} and the norm of $\tilde{\Phi}_{n}$ with respect to $d\mu$, we get
	\begin{flalign*}
		\nonumber	\int\lvert\mathbb{K}_n(z,s,\tilde{\mu})\rvert^2d\mu(s)&\leq\mathbb{K}_n(z,z,\tilde{\mu})\left[\sum_{j=0}^{n}\frac{1}{\lVert\tilde{\Phi}_{j}\rVert_{\tilde{\mu}}^2}\lVert\tilde{\Phi}_{j}\rVert_{\mu}^2\right]\\
		\nonumber	&\leq\mathbb{K}_n(z,z,\tilde{\mu})\sum_{j=0}^{n}\tilde{k}_j^2(1+\lvert a_j\rvert^2)m_o(\mu).
	\end{flalign*}
	This completes the proof.
\end{proof}
\begin{theorem}\label{LubinskytypeinequalityforM2class}
	Let $\mu \in \mathcal{M}_2$ be a positive Borel measure on the unit circle as defined in Definition \ref{Def. Quasi OPUC}. Suppose $\mathbb{K}_n(z,w,\mu)$ and $\mathbb{K}_n(z,w,\tilde{\mu})$ denote the CD kernels corresponding to the measure $\mu$ and $\tilde{\mu}$, respectively as given in \eqref{CD kernel polynomial}.  Then, for any complex number z and w, we have
	\begin{flalign*}
		\lvert\mathbb{K}_n(z,w,\mu)-&\mathbb{K}_n(z,w,\tilde{\mu})\rvert^2
		\\
		&\leq\mathbb{K}_n(w,w,\mu)\left[\mathbb{K}_n(z,z,\mu)-\mathbb{K}_n(z,z,\tilde{\mu})\left(2-\sum_{j=0}^{n}\tilde{k}_j^2(1+\lvert a_j\rvert^2)m_o(\mu)\right)\right],
	\end{flalign*}
where $\tilde{k}_n=\lVert\tilde{\Phi}_{n}\rVert_{\tilde{\mu}}^{-1}$ and $m_o(\mu)=\int_{0}^{2\pi}d\mu(e^{i\theta})$.
\end{theorem}
\begin{proof}
By the reproducing property of CD kernel, we can write
\begin{flalign*}
	\mathbb{K}_n(z,w,\mu)-\mathbb{K}_n(z,w,\tilde{\mu})=\int\left(\mathbb{K}_n(z,s,\mu)-\mathbb{K}_n(z,s,\tilde{\mu})\right)\mathbb{K}_n(s,w,\mu)d\mu(s),
\end{flalign*}
and using the Cauchy-Schwarz inequality, the above expression would be
\begin{flalign}\label{Int.of.sq.of.diff}
\nonumber	\lvert\mathbb{K}_n(z,w,\mu)-\mathbb{K}_n(z,w,\tilde{\mu})\rvert^2&\leq\int\lvert\mathbb{K}_n(z,s,\mu)-\mathbb{K}_n(z,s,\tilde{\mu})\rvert^2 d\mu(s)\int \lvert\mathbb{K}_n(s,w,\mu)\rvert ^2d\mu(s)\\
	&=\mathbb{K}_n(w,w,\mu)\underbrace{\left(\int\lvert\mathbb{K}_n(z,s,\mu)-\mathbb{K}_n(z,s,\tilde{\mu})\rvert^2 d\mu(s)\right)}_\text{I}.
\end{flalign}
Firstly, we simplify the bracketed term in \eqref{Int.of.sq.of.diff}.
\begin{flalign*}
	I=\int\mathbb{K}_n(z,s,\mu)\mathbb{K}_n(s,z,\mu)d\mu(s)&-\int\overline{\mathbb{K}_n(z,s,\tilde{\mu})}\mathbb{K}_n(z,s,\mu)d\mu(s)\\
	-\int\mathbb{K}_n(z,s,\tilde{\mu})&\mathbb{K}_n(s,z,\mu)d\mu(s)
+\int\mathbb{K}_n(z,s,\tilde{\mu})\overline{\mathbb{K}_n(z,s,\tilde{\mu})}d\mu(s),
\end{flalign*}
\begin{equation}\label{Simplified.Int.of.sq.of.diff}
\implies	I=\mathbb{K}_n(z,z,\mu)-\mathbb{K}_n(z,z,\tilde{\mu})-\overline{\mathbb{K}_n(z,z,\tilde{\mu})}+\int\mathbb{K}_n(z,s,\tilde{\mu})\overline{\mathbb{K}_n(z,s,\tilde{\mu})}d\mu(s).
\end{equation}
Substituting \eqref{L2norm.ineq.Kernel.of.quasi}, in  \eqref{Simplified.Int.of.sq.of.diff}, we have
\begin{equation}\label{Simplified.ineq.Int.of.sq.of.diff}
	I\leq\mathbb{K}_n(z,z,\mu)-2\mathbb{K}_n(z,z,\tilde{\mu})+\mathbb{K}_n(z,z,\tilde{\mu}) )\sum_{j=0}^{n}\tilde{k}_j^2(1+\lvert a_j\rvert^2)m_o(\mu).
\end{equation}
Hence, by using \eqref{Simplified.ineq.Int.of.sq.of.diff} in \eqref{Int.of.sq.of.diff}, we obtain
\begin{flalign*}
	\lvert\mathbb{K}_n(z,w,\mu)-\mathbb{K}_n(z,w,\tilde{\mu})\rvert^2
	&\leq\mathbb{K}_n(w,w,\mu)\\
	&\left[\mathbb{K}_n(z,z,\mu)-\mathbb{K}_n(z,z,\tilde{\mu})\left(2-\sum_{j=0}^{n}\tilde{k}_j^2(1+\lvert a_j\rvert^2)m_o(\mu)\right)\right].
\end{flalign*}
This completes the proof.
\end{proof}
\begin{remark}
	In one sense, Theorem \ref{LubinskytypeinequalityforM2class} is better than Theorem \ref{Lubinsky Inequality}, as the ordering between the $\mu$ and $\tilde{\mu}$ ($\mu\lessgtr\tilde{\mu}$) is not required. On the other hand, an extra expression
	\begin{align}\label{extra exp Lub.type}
	\left(2-\sum_{j=0}^{n}\tilde{k}_j^2(1+\lvert a_j\rvert^2)m_o(\mu)\right),
	\end{align}
	is presented in Theorem \ref{LubinskytypeinequalityforM2class}, which makes it weaker than Theorem \ref{Lubinsky Inequality}. This term can be discarded from the hypothesis to obtain the Lubinsky inequality without ordering between $\mu$ and $\tilde{\mu}$, if
	\begin{align}\label{need cond}
		\sum_{j=0}^{n} \frac{(1+\lvert a_j\rvert^2)m_o(\mu)}{\lVert\tilde{\Phi}_j\rVert^2_{\tilde{\mu}}}\leq1,
	\end{align}
	which is NOT easy to prove. This leads to the following question.
\end{remark}
\begin{problem}
	Is it possible to impose certain conditions on the coefficients $a_j$'s in \eqref{Quasi OPUC eq.} such that \eqref{need cond} is true?
\end{problem}
\subsection{Sub-reproducing property}
We observe that $\mathbb{K}_n(z, w, \mu)$ denotes the kernel polynomials for orthonormal polynomials. It is important to note that normalization plays a crucial role in defining kernel polynomials, providing the reproducing property \cite[eq 1.18]{B.SimonCDkernel2008} and CD formula. On the other hand, if we define the kernel polynomials for orthogonal polynomials, it will not yield the reproducing property. Nevertheless, we can still inquire about some estimates related to the reproducing type property for the kernel polynomial of orthogonal polynomials. We define
\begin{align}\label{kernel type polynomials}
	\mathcal{K}_n(z,w,\mu)=\sum_{j=0}^{n}\overline{\Phi_j(w)}\Phi_j(z).
\end{align}
We refer to $\mathcal{K}_n(z,w,\mu)$ as kernel-type polynomials. In the subsequent result, we obtain one-sided inequality for the reproducing property, which is called as sub-reproducing property.
\begin{proposition}[Sub-Reproducing property]
	Let $\mathcal{K}_n(z,w,\mu)$ denote the polynomial of degree $n$ in $z$, as defined in \eqref{kernel type polynomials}. If $p(z)$ is any polynomial of degree at-most $n$ such that $Re\left(\langle p, \Phi_j\rangle\right)\geq0$ for every $j=0,1,...,n$, then for any $w\in\mathbb{C}$ such that $\Phi_j(w)\geq0$ for each $j=0,1,...,n$, we have
\begin{align}\label{kernel type reproducing property}
Re\left(\int \overline{p(z)}\mathcal{K}_n(z,w,\mu)d\mu(z)\right)\leq m_0(\mu)Re(p(w)).
\end{align}
If $\mu$ is a probability measure on the unit circle then
\begin{align}\label{kernel type reproducing property for prob. measure}
	Re\left(\int \overline{p(z)}\mathcal{K}_n(z,w,\mu)d\mu(z)\right)\leq Re(p(w)).
\end{align}
\end{proposition}
\begin{proof}
	Since $p(z)=\sum_{j=0}^{n}\beta_j\Phi_{j}(z)$. For any $w\in\mathbb{C}$ such that $\Phi_j(w)\geq0$ for each $j=0,1,...,n$. We can write
	 	\begin{align*}
		Re\left(\int \overline{p(z)}\mathcal{K}_n(z,w,\mu)d\mu(z)\right)&=Re\left(\int\sum_{j=0}^{n}\overline{\beta_j}\overline{\Phi_{j}(z)}\sum_{i=0}^{n}\overline{\Phi_i(w)}\Phi_i(z)d\mu(z)\right)\\
		&=\sum_{j=0}^{n}Re(\overline{\beta_j\Phi_j(w))}\lVert\Phi_j\rVert_\mu^2\\
		&\leq m_0(\mu)Re(p(w)),
	\end{align*}
last inequality follows from the minimization property of $\Phi_n$ with respect to $\mu$.
 \end{proof}
\begin{corollary}
	Let $\mathcal{K}_n(z,\eta,\mu)$ denote the polynomial of degree $n$ in $\eta$, as defined in \eqref{kernel type polynomials}. Then for any $z\in \mathbb{C}$, we have
	\begin{align}\label{kernel type reproducing for kernel}
		\int \mathcal{K}_n(z,\eta, \mu)\mathcal{K}_n(\eta,z,\mu)d\mu(\eta)\leq m_0(\mu)\mathcal{K}_n(z,z, \mu),
	\end{align}
where	$m_o(\mu)=\int_{0}^{2\pi}d\mu(e^{i\theta})$.
\end{corollary}
\begin{proof}
	For any $z\in\mathbb{C}$, we write
	\begin{align*}
		\int \mathcal{K}_n(z,\eta, \mu)\mathcal{K}_n(\eta,z,\mu)d\mu(\eta)&=\sum_{j=0}^{n}\overline{\Phi_j(z)}\Phi_j(z)\lVert\Phi_j\rVert_\mu^2\\
		&\leq m_0(\mu)\mathcal{K}_n(z,z, \mu).
	\end{align*}
This completes the proof.
\end{proof}
In the next theorem, we obtain estimates for the absolute difference of diagonal elements of the kernel-type polynomials with respect to the measures $\mu$ and $\tilde{\mu}$.
\begin{theorem}\label{bound of difference of inverse Christoffel fn of mu and tildemu}
	For $z\in \overline{\mathbb{D}}$ and $\{\alpha_j\}$ representing the Verblunsky coefficients corresponding to the measure $\mu$, we have
	\begin{flalign*}
		\lvert\mathcal{K}_n(z,z,\tilde{\mu})-\mathcal{K}_n(z,z,\mu)\rvert\leq\sum_{j=0}^{n}\left(\left(\lvert \alpha_{j-1}\rvert+1\right)^2+2\lvert a_j\rvert^2\right)exp\left(2\sum_{k=0}^{j-2}\lvert\alpha_k\rvert\right).
	\end{flalign*}
	Moreover, if $\alpha_{n}\neq0$ for $n\geq1$ then
	\begin{flalign}\label{ineq.of.absolute.diff.of.kernel.and.quasi.kernel}
		\lvert\mathcal{K}_n(z,z,\tilde{\mu})-\mathcal{K}_n(z,z,\mu)\rvert\leq M+6\sum_{j=3}^{n}\frac{1}{\lvert\alpha_{j-2}\rvert^2}e^{2j-2},
	\end{flalign}
	where $M=2e^4(6+\sum\limits_{j=0}^{2}\lvert a_j\rvert^2)$ and $a_j$'s given in \eqref{Quasi OPUC eq.}.
\end{theorem}
\begin{proof} We can write the reproducing kernel with respect to the measure $\tilde{\mu}$ as
	\begin{fleqn}[\parindent]
		\begin{align}\label{expansion.of.Kernel.of.quasi.opuc}
			\nonumber		\mathcal{K}_n(z,w,\tilde{\mu})&=\sum_{j=0}^{n}(\Phi_{j}(z)-a_j\Phi_{j-1}(z))(\overline{\Phi_{j}(w)}-\overline{a_j}\overline{\Phi_{j-1}(w)})\\
			&=\sum_{j=0}^{n}\Phi_{j}(z)\overline{\Phi_{j}(w)}-\sum_{j=0}^{n}a_{j}\Phi_{j-1}(z)\overline{\Phi_{j}(w)}-\sum_{j=0}^{n}\overline{a_{j}}\Phi_{j}(z)\overline{\Phi_{j-1}(w)}\\
			\nonumber		&+\sum_{j=0}^{n}\lvert a_{j}\rvert^2\Phi_{j-1}(z)\overline{\Phi_{j-1}(w)},
		\end{align}	
	\end{fleqn}
	which implies
	\begin{flalign*}\mathcal{K}_n(z,w,\tilde{\mu})-\mathcal{K}_n(z,w,\mu)&=\sum_{j=0}^{n}\lvert a_{j}\rvert^2\Phi_{j-1}(z)\overline{\Phi_{j-1}(w)}-\sum_{j=0}^{n}a_{j}\Phi_{j-1}(z)\overline{\Phi_{j}(w)}\\
		+\sum_{j=0}^{n}&\overline{a_{j}}\overline{\Phi_{j-1}(z)}\Phi_{j}(w)-\sum_{j=0}^{n}\overline{a_{j}}\overline{\Phi_{j-1}(z)}\Phi_{j}(w)-\sum_{j=0}^{n}\overline{a_{j}}\Phi_{j}(z)\overline{\Phi_{j-1}(w)}\\
		=\sum_{j=0}^{n}\lvert a_{j}\rvert^2\Phi_{j-1}(z)\overline{\Phi_{j-1}(w)}&-\sum_{j=0}^{n}2Re(a_j\Phi_{j-1}(z)\overline{\Phi_{j}(w)})+\sum_{j=0}^{n}\overline{a_{j}}\overline{\Phi_{j-1}(z)}\Phi_{j}(w)\\
		&-\sum_{j=0}^{n}\overline{a_{j}}\Phi_{j}(z)\overline{\Phi_{j-1}(w)}.
	\end{flalign*}
	For $z=w$, we have
	
	\begin{flalign*}
		\mathcal{K}_n(z,z,\tilde{\mu})-\mathcal{K}_n(z,z,\mu)=\sum_{j=0}^{n}\lvert a_{j}\rvert^2\Phi_{j-1}(z)\overline{\Phi_{j-1}(z)}-\sum_{j=0}^{n}2Re(a_j\Phi_{j-1}(z)\overline{\Phi_{j}(z)}).
	\end{flalign*}
	Using the triangle inequality, we get
	\begin{flalign*}
		\lvert\mathcal{K}_n(z,z,\tilde{\mu})-\mathcal{K}_n(z,z,\mu)\rvert\leq\sum_{j=0}^{n}\lvert a_{j}\rvert^2\lvert\Phi_{j-1}(z)\rvert^2+\sum_{j=0}^{n}2\lvert
		Re(a_j\Phi_{j-1}(z)\overline{\Phi_{j}(z)})\rvert.
	\end{flalign*}
	We know that $2\lvert Re(zw)\rvert\leq\lvert z\rvert^2+\lvert w\rvert^2$. Hence,
	\begin{flalign*}
		\lvert\mathcal{K}_n(z,z,\tilde{\mu})-\mathcal{K}_n(z,z,\mu)\rvert\leq\sum_{j=0}^{n}\lvert\Phi_{j}(z)\rvert^2+2\sum_{j=0}^{n}
		\lvert a_j\rvert^2\lvert\Phi_{j-1}(z)\rvert^2.
	\end{flalign*}
	For $z\in \partial\mathbb{D}$, using the inequality\cite[eq. 1.5.19]{SimonOPUC1} $\lvert\Phi_{n+1}(z)\rvert\leq(1+\lvert\alpha_{n}\rvert)\lvert\Phi_{n}(z)\rvert$, we can write
	\begin{flalign*}
		\lvert\mathcal{K}_n(z,z,\tilde{\mu})-\mathcal{K}_n(z,z,\mu)\rvert\leq\sum_{j=0}^{n}\left(\left(\lvert \alpha_{j-1}\rvert+1\right)^2+2\lvert a_j\rvert^2\right)\lvert\Phi_{j-1}(z)\rvert^2,
	\end{flalign*}
	from  which again using the inequality \cite[eq. 1.5.17]{SimonOPUC1}, we obtain the desired result
	\begin{flalign*}
		\lvert\mathcal{K}_n(z,z,\tilde{\mu})-\mathcal{K}_n(z,z,\mu)\rvert\leq\sum_{j=0}^{n}\left(\left(\lvert \alpha_{j-1}\rvert+1\right)^2+2\lvert a_j\rvert^2\right)exp\left(2\sum_{k=0}^{j-2}\lvert\alpha_k\rvert\right).
	\end{flalign*}
	If $\alpha_{n}\neq0$ for $n\geq1$ then, using $a_{n+1}=\frac{\overline{\alpha_{n}}}{\overline{\alpha_{n-1}}}$ \cite[Theorem 4]{Branquinho_Golinskii_Paco-CVandEE-1999}, we have
	\begin{flalign*}
		\lvert\mathcal{K}_n(z,z,\tilde{\mu})-\mathcal{K}_n(z,z,\mu)\rvert&\leq\sum_{j=0}^{2}\left(\left(\lvert \alpha_{j-1}\rvert+1\right)^2+2\lvert a_j\rvert^2\right)exp\left(2\sum_{k=0}^{j-2}\lvert\alpha_k\rvert\right)\\
		&+\sum_{j=3}^{n}\left(\left(\lvert \alpha_{j-1}\rvert+1\right)^2+2\frac{\lvert\alpha_{j-1}\rvert^2}{\lvert\alpha_{j-2}\rvert^2}\right)exp\left(2\sum_{k=0}^{j-2}\lvert\alpha_k\rvert\right)\\
		&\leq M+6\sum_{j=3}^{n}\frac{1}{\lvert\alpha_{j-2}\rvert^2}e^{2j-2},
	\end{flalign*}
	where $M=2e^4(6+\sum\limits_{j=0}^{2}\lvert a_j\rvert^2)$.
\end{proof}
\section{Concluding Remarks}
 The relationship between the concept of a chain sequence and the OPRL is widely acknowledged. Specifically, the chain sequence $c_n$ is derived from the recurrence coefficients of the OPRL, expressed as $c_n=\frac{\lambda_n}{h_nh_{n+1}}$, where $h_n>0$ are diagonal elements and $\lambda_n$'s are subdiagonal elements of the corresponding Jacobi matrix of the OPRL. In the context of OPUC, the positive chain sequence is linked to the real Verblunsky coefficients $\alpha_n\in (-1,1)$ (as seen in \cite{Delsarte_Genin_Split Levinson_IEEE_1986}). More broadly, the connection between the positive chain sequence and complex Verblunsky coefficients is explored in \cite{Costa_Felix_Ranga_OPUC and chain seq_JAT_2013} through the derivation of the three-term recurrence relation of the ``modified CD kernel". In \eqref{Chain sequence of M2 class}, we express the positive chain sequence in terms of Verblunsky coefficients and $a_n$'s originating from the Marcell\'{a}n class $\mathcal{M}_2$. Notice that in \eqref{Chain sequence of M2 class}, $c_{n+1}=\frac{1}{4}$ for $n\geq3$. Hence for $n\geq3$, \eqref{TTRR of modified CD kernel which involves chain sequence} reduces to
\begin{align*}
	\mathcal{R}_{n+1}(z)=(z+1)\mathcal{R}_{n}(z)-z\mathcal{R}_{n-1}(z).
\end{align*}

It is noteworthy that Proposition \ref{mu tilde is ac} elucidates that even in the absence of the specific formulation provided in {\rm\cite[Theorem 14]{Branquinho_Golinskii_Paco-CVandEE-1999}} for $\tilde{\mu}$, it is possible to infer the absence of any singular component in the measure $\tilde{\mu}$ and the derivative of $\tilde{\mu}$ is in $L^2(0,2\pi)$. Furthermore, in Subsection \ref{subsec: Involving norm inequalities}, we have derived lower and upper bounds for the norm of polynomials with respect to measures $\mu$ and $\tilde{\mu}$ in terms of $a_n$'s.

If Proposition \ref{int PhiN non zero wrt tildemu} can be independently demonstrated, without relying on the condition $a_n \neq 0$ for all $n \in \mathbb{N}$, an alternative proof regarding the non-zero nature of $a_n$'s can be established. In other words, when $\mu$ belongs to the class $\mathcal{M}_2$, there is no possible existence of any $N \in \mathbb{N}$ for which $\tilde{\Phi}_N(z) = \Phi_N(z)$ holds true, achieved through the utilization of the orthogonality property of $\tilde{\Phi}_n$ with respect to $\tilde{\mu}$.

To relax the assumptions of the Lubinsky inequality, particularly concerning the pair of measures ($\mu, \tilde{\mu}$), and examine the behavior of the off-diagonal CD kernel compared to its diagonal counterpart, we present Theorem \ref{LubinskytypeinequalityforM2class}. Section \ref{sec:Lubinsky type inequality} concludes by defining kernel-type polynomials and establishing the sub-reproducing property for them. Theorem \ref{bound of difference of inverse Christoffel fn of mu and tildemu} demonstrates that when dealing with a sequence of non-zero Verblunsky coefficients, it becomes possible to achieve a bound for the absolute difference between the kernel-type polynomials of measures $\mu$ and $\tilde{\mu}$ in terms of these Verblunsky coefficients.

\textbf{Acknowledgement}: The work of the second author is supported by the NBHM(DAE) Project No. NBHM/RP-1/2019, Department of Atomic Energy, Government of India.

	

\begin{thebibliography}{55}
		\bibitem{When do linear}M. Alfaro, F. Marcellán, A. Pe$\tilde{\text{n}}$a and M. Luisa Rezola,  When do linear combinations of orthogonal polynomials yield new sequences of orthogonal polynomials?, J. Comput. Appl. Math.  233 (2010), no.~6, 1446--1452.
		\bibitem{Behera_Ranga_Swami_SIGMA_2016}K.K. Behera, A. Sri Ranga, A. Swaminathan, Orthogonal polynomials associated with complementary chain sequences, SIGMA 12 (2016) 075, 17 pages.
		\bibitem{Behera_Swami_Missing term_CMFT_2018}K.K. Behera, A. Swaminathan, Orthogonal polynomials related to g-fractions with missing terms, Comput. Methods Funct.Theory 18(2) (2018) 193–219.
		\bibitem{Brac_Finkel_Ranga_Veronese_2018_JAT}C. F. Bracciali, A. Martínez-Finkelshtein, A. Sri Ranga, and D. O. Veronese, Christoffel formula for kernel polynomials on the unit circle, J. Approx. Theory 235 (2018), 46--73.
		\bibitem{Branquinho_Golinskii_Paco-CVandEE-1999}A. Branquinho, L. Golinskii and F. Marcellán,
		Orthogonal polynomials and rational modifications of Lebesgue measure on the unit circle. An inverse problem. Complex Variables Theory Appl. 38 (1999), no.2, 137–154.
		\bibitem{Branquinho_Paco_1996_Generating}A. Branquinho and F. Marcell\'{a}n, Generating new classes of orthogonal polynomials,
		Inter. J. Math. Math. Sci. 19(4) (1996), pp. 643-656.
		\bibitem{Bultheel_Quasi_POPUC_JCAM_2022} A. Bultheel, R. Cruz-Barroso\ and\ C. D\'{\i}az~Mendoza, Zeros of quasi-paraorthogonal polynomials and positive quadrature, J. Comput. Appl. Math. {\bf 407} (2022), Paper No. 114039, 22 pp., https://doi.org/10.1016/j.cam.2021.114039.
		\bibitem{Cacha_Paco_Proc. AMS_1998_CLC of OPS}A. Cachafeiro, F. Marcellán, Convex linear combinations of sequences of monic orthogonal polynomials, Proc. Amer. Math. Soc. 126 (8) (1998) 2323--2331.
		\bibitem{CMV_Para OPUC_JAT_2002}M.J. Cantero, L. Moral, L. Velázquez, Measures and para-orthogonal polynomials on the unit circle, East J. Approx. 8 (2002) 447--464.
		\bibitem{Chihara book}T. S. Chihara, {\it An introduction to orthogonal polynomials}, Mathematics and its Applications, Vol. 13, Gordon and Breach Science Publishers, New York, 1978.
		\bibitem{Costa_Felix_Ranga_OPUC and chain seq_JAT_2013}M.S. Costa, H.M. Felix, A. Sri Ranga, Orthogonal polynomials on the unit circle and chain sequences, J. Approx. Theory 173 (2013) 14--32.
		\bibitem{Paco_ST for HTM_JCAM_2007}L. Darius, J. Hernández, F. Marcellán, Spectral transformations for Hermitian Toeplitz matrices, J. Comput. Appl. Math. 202 (2007) 155--176.
		\bibitem{Delsarte_Genin_Split Levinson_IEEE_1986}P. Delsarte, Y. Genin, The split Levinson algorithm, IEEE Trans. Acoust. Speech Signal Process. 34 (1986) 470--478.
		\bibitem{Driver_Jordaan_SIGMA_2016}K. Driver, K. Jordaan, Zeros of quasi-orthogonal Jacobi polynomials, SIGMA 12 (2016) 042, 13 pages.
		\bibitem{Garza_CT for matrix measure_CMFT_2021}H. Dueñas,  E. Fuentes,  L. E. Garza, On a Christoffel Transformation for Matrix Measures Supported on the Unit Circle, Comput. Methods Funct. Theory 21, (2021) 219–243.
		\bibitem{Geronimus-OPUC and their applications-transl. (1962)} Ya L. Geronimus, Polynomials Orthogonal on a circle and their applications, Amer. Math.
		Soc. Transl. (1), vol. 3, Providence, Rhode Island (1962), 1-78.
		\bibitem{Njastad_1996_CD kernel and para-OPUC}P. González-Vera, J.C. Santos-León, O. Njåstad, Some results about numerical quadrature on the unit circle, Adv. Comput. Math. 5 (1996) 297--328.
	\bibitem{Ismail_Ranga_GEVP_OPUC_Linear algebra_2019}	M.~E.-H. Ismail\ and\ A. Sri~Ranga, $R_{II}$ type recurrence, generalized eigenvalue problem and orthogonal polynomials on the unit circle, Linear Algebra Appl. {\bf 562} (2019), 63--90.
		\bibitem{Vikas_Swami_quasi-type kernel} V. Kumar and A. Swaminathan, Recovering the orthogonality from quasi-type kernel polynomials using specific spectral transformations, arXiv preprint, arXiv:2211.10704v3, (2023), 25 pages.
		\bibitem{Lubinsky2007annals}D. S. Lubinksy, A new approach to universality limits involving orthogonal polynomials, Annals. of Math. 170(2009), 915-939.
		\bibitem{Paco_Maroni_OPUC derivate orthogonal NS_1991_Const. app.}F. Marcellán, P. Maroni, Orthogonal polynomials on the unit circle and their derivatives, Constr. Approx. 7 (1991) 341--348.
		\bibitem{Paco_Peherstorfer_Orthogonality of FLOP_ACM_1996} F. Marcell\'{a}n, F. Peherstorfer\ and\ R. Steinbauer, Orthogonality properties of linear combinations of orthogonal polynomials, Adv. Comput. Math. {\bf 5} (1996), no.~4, 281--295.
		\bibitem{Nevai_Case study Christoffel functions_1986}P. Nevai, 	Géza Freud, orthogonal polynomials and Christoffel functions. A case study. J. Approx. Theory 48(1986), no.1, 3–167.
		\bibitem{Peherstorfer_1990_MOC_quasi} F. Peherstorfer, Linear combinations of orthogonal polynomials generating quadrature formulas, Math. Comp. 55 (1990) 231--241.
		\bibitem{Qiu_Ma_Chu_CMFT_2022}S.L. Qiu, X.Y. Ma, Y.M. Chu, Transformation Properties of Hypergeometric Functions and Their Applications. Comput. Methods Funct. Theory 22 (2022) 323–366.
		\bibitem{Shukla_Swami_Malaysian_2023}V. Shukla\ and\ A. Swaminathan, Spectral transformation associated with a perturbed $R_I$ type recurrence relation, Bull. Malays. Math. Sci. Soc. {\bf 46} (2023), no.~5, Paper No. 169, 29 pp., https://doi.org/10.1007/s40840-023-01561-8.
		\bibitem{Simanek_Exceptional OP_CMFT_2023}B. Simanek, Convergence Rates of Exceptional Zeros of Exceptional Orthogonal Polynomials. Comput. Methods Funct. Theory 23, (2023) 629–649.
		\bibitem{SimonOPUC1}B. Simon, Orthogonal Polynomials on the Unit Circle. Part 1. Classical Theory. Part 1, Amer. Math. Soc. Colloq. Publ., vol.54, Amer. Math. Soc., Providence, RI, 2005.
		\bibitem{B.SimonCDkernel2008}B. Simon, The Christoffel-Darboux kernel, Perspectives in partial differential equations, harmonic analysis and applications, Proc. Sympos. Pure Math., vol. 79, Amer. Math. Soc., Providence, RI, 2008, pp. 295--335.
		\bibitem{B.SimonSzegoDescendants}B. Simon, Szegő's Theorem and its Descendants. Spectral Theory for $L^2$ Perturbations of Orthogonal Polynomials, Princeton University Press, Princeton, NJ, 2011.
		\bibitem{Riesz} M. Riesz,  Sur le probl\'{e}me des moments, Troisième note, Ark. Mat. Astr. Fys., 17 (1923), 1-52.
		\bibitem{Zhedanov_SIGMA_2020} A. Zhedanov, An explicit example of polynomials orthogonal on the unit circle with a dense point spectrum generated by a geometric distribution,
		SIGMA 16 (2020) 140, 9 pages.
	\end{thebibliography}
\end{document}